\documentclass{amsart}
\usepackage[all]{xy}
\usepackage{amsmath}
\usepackage{amssymb}
\usepackage{amsthm}
\usepackage{amscd}
\newcounter{pcounter}

\newcommand{\R}{\mathcal{R}}

\newcommand{\F}{\mathcal{F}}

\newcommand{\B}{\mathcal{B}}

\newcommand{\PP}{\Bbb P}

\newcommand{\RR}{\Bbb R} 

\newcommand{\NN}{\Bbb N}

\newcommand{\CC}{\Bbb C}
\newcommand{\Lip}{Lip}

\newcommand{\ip}[1]{\langle #1 \rangle}

\newcommand{\widetidle}{\widetilde}
\newcommand{\varespilon}{\varepsilon}
\newcommand{\parital}{\partial}
\newcommand{\actson}{\curvearrowright}
\newtheorem{ex}{Example}

\newtheorem{theorem}{Theorem}
\newtheorem{definition}[theorem]{Definition}
\newtheorem{proposition}[theorem]{Proposition}
\newtheorem{cor}[theorem]{Corollary}
\newtheorem{lemma}[theorem]{Lemma}

\newcommand{\FF}{\Bbb F}

\DeclareMathOperator{\conjugate}{conj}
\DeclareMathOperator{\multi}{multi}

\DeclareMathOperator{\Hamm}{Hamm}
\DeclareMathOperator{\vol}{vol}
\DeclareMathOperator{\vr}{vr}
\DeclareMathOperator{\Isom}{Isom}
\DeclareMathOperator{\Sym}{Sym}

\DeclareMathOperator{\Span}{Span}

\DeclareMathOperator{\im}{im}
\DeclareMathOperator{\id}{Id}

\DeclareMathOperator{\Hom}{Hom}

\DeclareMathOperator{\Tr}{Tr}

\DeclareMathOperator{\opdim}{opdim}

\numberwithin{theorem}{section}

\begin{document}
\title[$l^{p}$ Dimension for Banach Space Representations of Sofic Groups] {An $l^{p}$-Version of Von Neumann Dimension for Banach Space Representations of Sofic Groups }        
\author{Ben Hayes}
\address{UCLA Math Sciences Building\\
         Los Angeles,CA 90095-1555} 
\email{brh6@ucla.edu}       
\date{\today}          
\maketitle
\begin{abstract} In \cite{Gor}, A.Gournay defined  a notion of $l^{p}$-dimension for $\Gamma$-invariant subspaces of $l^{q}(\Gamma)^{\oplus n},$ with $\Gamma$ amenable. The number $\dim_{l^{q}}l^{p}(\Gamma,V)$ is $\dim V$ when $p=q,$ and is preserved by a certain class of $\Gamma$-equivariant bounded linear isomorphisms.  We develop a notion of $\dim_{l^{p},\Sigma}(Y,\Gamma)$ where $Y$ is a Banach space with a uniformly bounded action of a sofic group $\Gamma$ and $\Sigma$ is a sofic approximation. In particular, our definition makes sense for a large class of non-amenable groups. We also develop a notion of $\dim_{S^{p},\Sigma}(Y,\Gamma)$ with $\Gamma$ a $\R^{\omega}$-embeddable group and $S^{p}$ the space of  finite dimensional Schatten $p$-class operators. These numbers are invariant under bounded $\Gamma$-equivariant linear isomorphisms and under the natural translation action of $\Gamma,$  $\dim_{l^{p}}(l^{p}(\Gamma,V),\Gamma)=\dim V,$ and $ \dim_{S^{p}}(l^{p}(\Gamma,V),\Gamma)=\dim V$ for $1\leq p\leq 2.$  In particular, this shows that $l^{p}(\Gamma,V)$ is not isomorphic to $l^{p}(\Gamma,W)$ as a representation of $\Gamma$ if $\dim V\ne \dim W,$ and $\Gamma$ is $\R^{\omega}$-embeddable. We discuss other concrete computations in a follow-up paper, including proving that our dimension agrees with von Neumann dimension for representations contained in a multiple of the left-regular representation.
\end{abstract}

\tableofcontents

\section{Introduction}
Let $\Gamma$ be a countable discrete group. Suppose that $H$ is a closed $\Gamma$-invariant subspace of $l^{2}(\Gamma\times \NN),$ and let $P_{H}$ be the projection onto $H,$ then it is known that the number
\[\dim_{L(\Gamma)}(H)=\sum_{n\in \NN}\ip{P_{H}\delta_{(e,n)},\delta_{(e,n)}}\]
obeys the usual properties of dimension,
\begin{list}{Property \arabic{pcounter}:~}{\usecounter{pcounter}}
\item$\dim_{L(\Gamma)}(H)=\dim_{L(\Gamma)}(K),$ if there is a $\Gamma$-equivariant bounded linear bijection from $H$ to $K$,\\
\item $\dim_{L(\Gamma)}(H\oplus K)=\dim_{L(\Gamma)}(H)+\dim_{L(\Gamma)}(K).$,\\
\item $\dim_{L(\Gamma)}(H)=0$ if and only if $H=0,$\\
\item $\dim_{L(\Gamma)}\left(\bigcap_{n=1}^{\infty}H_{n}\right)=\lim_{n\to \infty}\dim_{L(\Gamma)}(H_{n}),$ \mbox{ if $\dim_{L(\Gamma)}(H_{1})<\infty, $} and also $H_{n+1}\subseteq H_{n}$,\\
\item $\dim_{L(\Gamma)}\overline{\left(\bigcup_{n=1}^{\infty}H_{n}\right)}=\lim_{n\to \infty}\dim_{L(\Gamma)}(H_{n})$\mbox{ if $H_{n}\subseteq H_{n+1},$}
\end{list}
We also have
\[\dim_{L(\Gamma)}(l^{2}(\Gamma)^{\oplus n}))=n,\]

	Voiculescu in \cite{Voi} and Gournay in \cite{Gor} noticed that for \emph{amenable} groups $\Gamma,$ we can define this dimension  as a limit of normalized approximate dimensions of $F_{n}\Omega,$ with $F_{n}$ a F\o lner sequence, and $\Omega\subseteq H.$ This formula is analogous to the definition of entropy for actions of an amenable group on a compact metrizable space or measure space.  Gournay  noted that a formula for von Neumann dimension similar to Voiculescu's makes senses for subspaces of $l^{p}(\Gamma,V),$ with $\Gamma$ amenable. Using this, he defined an isomorphism invariant for subspaces of $l^{p}(\Gamma,V)$ agreeing with von Neumann dimension in the case $p=2$. In particular, Gournay shows that if $\Gamma$ is amenable, and there is an injective $\Gamma$-equivariant linear map of finite type (see \cite{Gor} for the definition) with closed image from $l^{p}(\Gamma,V)\to l^{p}(\Gamma,W)$ then $\dim V\leq \dim W.$ 

	Recently, in \cite{Bow},\cite{KLi} a theory of entropy for actions of a \emph{sofic} group on a probability space or a compact metrizable space has been developed. Using this theory, it was shown for sofic groups $\Gamma$ that probability measure preserving Bernoulli actions $\Gamma\curvearrowright (X,\mu)^{\Gamma},\Gamma \actson (Y,\nu)$ are not isomorphic if the entropy of $(X,\mu)$ does not equal the entropy of $(Y,\nu),$ if $\Gamma,$ and that Bernoulli actions $\Gamma\actson X^{\Gamma},\Gamma\actson Y^{\Gamma}$ are not isomorphic as actions on compact metrizable spaces if $|X|\ne |Y|$ (here $X$ and $Y$ are finite). We can think of the action of $\Gamma$ on $l^{p}(\Gamma,V)$ as analogous to a Bernoulli action, since both actions are given by translating functions on the group. Combining  ideas of Kerr and Li \cite{KLi} and Voiculescu in \cite{Voi}, we define an isomorphism invariant
\[\dim_{\Sigma,l^{p}}(Y,\Gamma)\]
for a uniformly bounded action of a sofic group on a separable Banach space $Y$ (all our Banach spaces will be complex, unless explicitly mentioned otherwise). 

	This definition of dimension has the following properties:
\begin{list}{Property \arabic{pcounter}:~}{\usecounter{pcounter}}
\item $\dim_{\Sigma,l^{p}}(Y,\Gamma)\leq \dim_{\Sigma,l^{p}}(X,\Gamma)$ if there is an equivariant bounded linear map $X\to Y$ with dense image,
\item $\dim_{\Sigma,l^{p}}(V,\Gamma)\leq \dim_{\Sigma,l^{p}}(W,\Gamma)+\dim_{\Sigma,l^{p}}(V/W,\Gamma),$ if $W\subseteq V$ is a closed $\Gamma$-invariant subspace,
\item $\dim_{\Sigma,l^{p}}(Y\oplus W,\Gamma)\geq \dim_{\Sigma,l^{p}}(Y,\Gamma)+\underline{\dim}_{\Sigma,l^{p}}(W,\Gamma)$ for $2\leq p<\infty,$ where $\underline{\dim}$ is a ``lower dimension," and is also an invariant, 
\item $\dim_{\Sigma,l^{p}}(l^{p}(\Gamma,V),\Gamma)=\underline{\dim}_{\Sigma,l^{p}}(l^{p}(\Gamma,V),\Gamma)=\dim(V)$ for $1\leq p\leq 2,$
\item $\dim_{\Sigma,l^{p}}(X,\Gamma)\geq \dim_{L(\Gamma)}(\overline{X}^{\|\cdot\|_{2}}),$ when $X\subseteq l^{p}(\NN,l^{p}(\Gamma))$ and $1\leq p\leq 2.$

\end{list}

We also note that for defining $\dim_{l^{p}}(Y,\Gamma),$ little about soficity of $\Gamma$ is used, and we can more generally define our invariants associated to a sequence of maps $\sigma_{i}\colon \Gamma\to \Isom(V_{i})$ where  $V_{i}$ are finite-dimensional Banach spaces.

	In particular, we can show that $\dim_{\Sigma,l^{2}}(Y,\Gamma)$ can be defined for $\R^{\omega}$-embeddable groups $\Gamma.$ Because unitaries also act isometrically on the space of Schatten $p$-class operators, we can also define an invariant
\[\dim_{\Sigma,S^{p}}(Y,\Gamma),\]
$S^{p}$ dimension has properties analogous to $l^{p}$ dimension.
\begin{list}{Property \arabic{pcounter}: ~ }{\usecounter{pcounter}}
\item $\dim_{\Sigma,S^{p}}(Y,\Gamma)\leq \dim_{\Sigma,S^{p}}(X,\Gamma)$ if there is a $\Gamma$-equivariant bounded linear bijection $X\to Y,$
\item $\dim_{\Sigma,S^{p}}(V,\Gamma)\leq \dim_{\Sigma,S^{p}}(W,\Gamma)+\dim_{\Sigma,S^{p}}(V/W,\Gamma),$ if $W\subseteq V$ is a closed $\Gamma$-invariant subspace,
\item $\dim_{\Sigma,S^{p}}(Y\oplus W,\Gamma)\geq \dim_{\Sigma,S^{p}}(Y,\Gamma)+\underline{\dim}_{\Sigma,S^{p}}(W,\Gamma)$ for $2\leq p<\infty,$ 
\item $\underline{\dim}_{\Sigma,S^{p}}(l^{p}(\Gamma,V),\Gamma)=\dim (V)$ for $1\leq p\leq 2,$\\
\item $\underline{\dim}_{\Sigma,S^{p}}(W,\Gamma)\geq \dim_{L(\Gamma)}(\overline{W}^{\|\cdot\|_{2}})$ if $W\subseteq l^{p}(\NN,l^{p}(\Gamma))$ is a nonzero closed invariant subspace and $1\leq p\leq 2,$
\item $\underline{\dim}_{\Sigma,l^{2}}(H,\Gamma)=\dim_{\Sigma,l^{2}}(H,\Gamma)=\dim_{L(\Gamma)}H$ if $H\subseteq l^{2}(\NN,l^{2}(\Gamma))$ is $\Gamma$ invariant.

\end{list}
In particular $l^{p}(\Gamma,V)$ is not isomorphic to $l^{p}(\Gamma,W)$ as a representation of $\Gamma,$ if $\Gamma$ is $\R^{\omega}$-embeddable and  $1\leq p<\infty.$ This extends a result of \cite{Gor} from  amenable groups to $\R^{\omega}$-embeddable groups, and answers a question of Gromov (see \cite{Gro} page 353) in the case of $\R^{\omega}$-embeddable groups.

	In a follow-up paper, we will do some more concrete computations:  defining $l^{p}$-Betti numbers for cocompact actions on CW complexes, and compute $l^{p}$-Betti numbers of free groups, as well computations for the natural action of $\Gamma$ on $L^{p}(L(\Gamma),\tau).$ We also prove that $\dim_{\Sigma,l^{p}}(X,\Gamma)=0$ if $X$ is finite-dimensional and $\Gamma$ is infinite.

\section{Definition of the Invariants}

	 We recall the definition of sofic and $\R^{\omega}$-embeddable groups (see \cite{Pest},\cite{Bow}) . To fix notation we use $\Sym(A)$ for the group of bijections of the set $A,$ and we let $S_{n}=\Sym(\{1,\cdots,n\}),$ finally we let $U(n)$ denote the unitary group of $\CC^{n},$ where $\CC^{n}$ has the usual inner product.  It is useful to introduce metrics on the symmetric and unitary groups. For $\sigma,\tau\in S_{n},$ we define the \emph{Hamming distance} by
\[d_{\Hamm}(\sigma,\tau)=\frac{1}{n}|\{j:\sigma(j)\ne \tau(j)\}|.\]
If $A,B\in M_{n}(\CC)$ we let
\[\ip{A,B}=\frac{1}{n}\Tr(B^{*}A),\]
note that $\ip{A,B}$ is indeed an inner product. We let $\|\cdot\|_{2}$ denote the Hilbert space norm induced by this inner product. 

\begin{definition}\label{D:sofic} \emph{ Let $\Gamma$ be a countable group. A} sofic approximation for $\Gamma$ \emph{ is a sequence of maps $\sigma_{i}\colon \Gamma\to S_{d_{i}}$ with $d_{i}\to \infty,$ (not assumed to be homomorphisms) which is approximately multiplicatively and approximately free in the sense that}
\[d_{\Hamm}(\sigma_{i}(st),\sigma_{i}(s)\sigma_{i}(t))\to 0, \mbox{\emph{ for all $s,t\in \Gamma$}}\]
\[d_{\Hamm}(\sigma_{i}(s),\sigma_{i}(s'))\to 1\mbox{\emph{ for all $s\ne s'\in \Gamma$} }.\]
\end{definition}
We say that $\Gamma$ is \emph{sofic} if it has a sofic approximation.

One can think of a sofic approximation $\sigma_{i}$ as above as maps so that if 
\[x_{1},\ldots,x_{n},y_{1},\ldots,y_{m}\in \Gamma,\]
 and $a_{1},\ldots,a_{n},b_{1},\ldots,b_{m}\in \{-1,1\},$ then with high probability,
\[\sigma_{i}(x_{1})^{a_{1}}\cdots \sigma_{i}(x_{n})^{a_{n}}(j)=\sigma_{i}(y_{1})^{a_{1}}\cdots \sigma_{i}(y_{n})^{a_{n}}(j)\mbox{ if $x_{1}^{a_{1}}\cdots x_{n}^{a_{n}}=y_{1}^{a_{1}}\cdots y_{n}^{a_{n}}$,}\]
\[\sigma_{i}(x_{1})^{a_{1}}\cdots \sigma_{i}(x_{n})^{a_{n}}(j)\ne\sigma_{i}(y_{1})^{a_{1}}\cdots \sigma_{i}(y_{n})^{a_{n}}(j) \mbox{ if $x_{1}^{a_{1}}\cdots x_{n}^{a_{n}}\ne y_{1}^{a_{1}}\cdots y_{n}^{a_{n}}$.}\]

The requirement $d_{i}\to \infty$ is not necessary since one can replace $\sigma_{i}$ with $\sigma_{i}^{\otimes k_{i}}$ where $\sigma_{i}^{\otimes k_{i}}\colon \Gamma\to \Sym(\{1,\ldots,d_{i}\}^{k_{i}})$ is given by
\[\sigma_{i}^{\otimes k_{i}}(s)(a_{1},\ldots,a_{k_{i}})=(\sigma_{i}(s)(a_{1}),\ldots,\sigma_{i}(s)(a_{k_{i}})).\]
We require that $d_{i}\to \infty$ simply for our properties of $l^{p}$-dimension to behave appropriately. Note that $d_{i}\to \infty$ is automatic when the group is infinite by our approximate freeness assumption.

A related notion is that of being $\R^{\omega}$-embeddable.

\begin{definition}\label{D:hyperlinear} \emph{ Let $\Gamma$ be a countable group. An} embedding sequence for $\Gamma$ \emph{ is a sequence of maps $\sigma_{i}\colon \Gamma\to U(d_{i}),$ with $d_{i}\to \infty,$ (not assumed to be homomorphisms) such that}
\[\|\sigma_{i}(st)-\sigma_{i}(s)\sigma_{i}(t)\|_{2}\to 0\mbox{ for all $s,t\in \Gamma$}\]
\[\frac{1}{d_{i}}\Tr(\sigma_{i}(s')^{*}\sigma_{i}(s))\to 0 \mbox{ for all $s\ne s'$ in $\Gamma.$}\]
\emph{ A group is said to be $\R^{\omega}$-embeddable if it has a embedding sequence.}
\end{definition}

	The second condition says that if $s\ne s'$, then asymptotically $\sigma_{i}(s),\sigma_{i}(s')$ become orthogonal  under the inner product which induces $\|\cdot\|_{2}.$ One can formulate a probabilistic interpretation of an embedding sequence analogous to that of a sofic approximation: for any $\varepsilon>0,$ if $x_{1},\ldots,x_{n},y_{1},\ldots,y_{m}\in \Gamma,$ and $a_{1},\ldots,a_{n},b_{1},\ldots,b_{m}\in \{-1,1\},$ then if $x_{1}^{a_{1}}\cdots x_{n}^{a_{n}}=y_{1}^{a_{1}}\cdots y_{n}^{a_{n}}$,
\[\PP(\{\xi\in S^{2d_{i}-1}:\|\sigma_{i}(x_{1})^{a_{1}}\cdots \sigma_{i}(x_{n})^{a_{n}}(\xi)-\sigma_{i}(y_{1})^{a_{1}}\cdots \sigma_{i}(y_{n})^{a_{n}}(\xi)\|<\varepsilon\})\to 1,\]
 and if if $x_{1}^{a_{1}}\cdots x_{n}^{a_{n}}\ne y_{1}^{a_{1}}\cdots y_{n}^{a_{n}},$ 
\[\PP(\{\xi\in S^{2d_{i}-1}:|\ip{\sigma_{i}(x_{1})^{a_{1}}\cdots \sigma_{i}(x_{n})^{a_{n}}(\xi),\sigma_{i}(y_{1})^{a_{1}}\cdots \sigma_{i}(y_{n})^{a_{n}}(\xi)}|<\varepsilon\})\to 1.\]
 This equivalence follows by concentration of measure.

Note that if $\sigma\in S_{n}$ and $U_{\sigma}$ is the unitary on $\CC^{n}$ which $\sigma$ induces, we have that
\[d_{\Hamm}(\sigma,\tau)=d_{\Hamm}(\tau^{-1}\sigma,\id)=1-\frac{1}{n}\Tr(U_{\tau^{-1}\sigma})=1-\frac{1}{n}\Tr(U_{\tau}^{*}U_{\sigma}),\]
\[\|U_{\sigma}-U_{\tau}\|_{2}^{2}=2-2(1-d_{\Hamm}(\tau^{-1}\sigma,\id))=2d_{\Hamm}(\sigma,\tau)\]
thus all sofic groups are $\R^{\omega}$-embeddable. 

	We will sometimes use an alternate definition of $\R^{\omega}$-embeddable: a group is $\R^{\omega}$-embeddable if its group von Neumann algebra embeds into an ultraproduct of matrix algebras. We will prove this in section \ref{S:lp}. For a good introduction to sofic and $\R^{\omega}$-embeddable groups we  refer to \cite{Pest}.

	We now give examples of sofic groups, and thus $\R^{\omega}$-embeddable groups, see (\cite{Pest} for proofs, although most of these can be shown directly).

\begin{ex}\label{E:am}\emph{ All countable amenable groups are sofic. To prove this, let $F_{n}$ is a F\o lner sequence for $\Gamma.$ For $g\in \Gamma,$ let $\tau_{i}(g)\colon F_{i}\setminus g^{-1}F_{i}\to F_{i}\setminus gF_{i}$ be an arbitrary bijection. Define  $\sigma_{i}\colon \Gamma \to \Sym(F_{i})$  by }\end{ex}
\[\sigma_{i}(s)(x)=\begin{cases}
  sx & \textnormal{ if $x\in F_{i}\cap s^{-1}F_{i}$}\\
  \tau_{i}(s)(x)  & \textnormal{ otherwise}
\end{cases}.\]
It follows directly form the definition of a F\o lner sequence that $\sigma_{i}$ is a sofic approximation.
\begin{ex}\label{E:fin} \emph{ All countable residually sofic groups are sofic. In particular, this includes all free groups and residually amenable groups.}\end{ex}
\begin{ex}\label{E:loc} \emph{ Countable locally sofic groups are sofic.}\end{ex}
\begin{ex}\emph{ By Malcev's Theorem (see \cite{BO} Theorem 6.4.3) all finitely generated linear groups are residually finite, hence sofic. By the preceding example all countable linear groups are sofic.} \end{ex}

	It is shown in \cite{ESZ1} that sofic groups are closed under direct products, taking subgroups, inverse limits, direct limits, free products, and extensions by amenable groups: if $\Lambda\triangleleft \Gamma,$ $\Lambda$ is sofic, and $\Gamma/\Lambda$ is amenable, then $\Gamma$ is sofic. It is also known that $\R^{\omega}$-embeddable groups are closed under these operations as well. It is unknown whether all countable groups are sofic. As mentioned earlier, a group is $\R^{\omega}$-embeddable if and only if its group von Neumann algebra embeds into an ultrapower of the hyperfinite $\rm{II}_{1}$ factor. It follows that if the Connes Embedding Conjecture is true, then all countable discrete groups are $\R^{\omega}$-embeddable. Even without the Connes Embedding conjecture we still have many examples of $\R^{\omega}$-embeddable groups.

\begin{definition} \emph{Let $X$ be a Banach space. An action $\Gamma$ on $X$ by is said to be} uniformly bounded \emph{if there is a constant $C>0$ such that}
\[\|sx\|\leq C\|x\|\mbox{ \emph{for all $x\in X,s\in \Gamma$}}.\]
\emph{We say that a sequence $S=(x_{j})_{j=1}^{\infty}$ in $X$ is} dynamically generating,  \emph{ if $S$ is bounded and $\Span\{sx_{j}:s\in \Gamma,j\in \NN\}$ is dense. }
\end{definition}

If $X$ is a Banach space we shall write $\Isom(X)$ for the group of all linear isometries from $X$ to itself.

\begin{definition} \emph{Let $V$ be a vector space with a pseudonorm $\rho.$ If $A\subseteq V,$ a linear subspace $W\subseteq V$ is said to} $\varepsilon$-contain $A,$ \emph{denoted $A\subseteq_{\varepsilon}W,$  if for every $v\in A,$ there is a $w\in W$ such that $\rho(v-w)<\varepsilon.$ We let $d_{\varepsilon}(A,\rho)$ be the minimal dimension of a subspace which $\varepsilon$-contains} $A.$\end{definition}

\begin{definition}\emph{A} dimension triple \emph{is a triple $(X,\Gamma,\Sigma=(\sigma_{i}\colon \Gamma\to \Isom(V_{i}))),$ where $X$ is a separable Banach space, $\Gamma$ is a countable discrete group with a uniformly bounded action on $X,$  each $V_{i}$ is finite-dimensional, and the $\sigma_{i}$ are functions with no structure assumed on them.}\end{definition}
\begin{definition} \emph{Let $(X,\Gamma,\Sigma=(\sigma_{i}\colon \Gamma\to \Isom(V_{i})))$ be a dimension triple.  Fix $S=(x_{j})_{j=1}^{\infty}$ a dynamically generating sequence in $X.$} 

	\emph{For $e\in E\subseteq \Gamma$ finite, $l\in \NN$ let}
\[X_{E,l}=\Span\{sx_{j}:s\in E^{l},1\leq j\leq l\}.\]

	\emph{If $e\in F\subseteq \Gamma$ finite, $m\in \NN,$ $C,\delta>0,$ let $\Hom_{\Gamma}(S,F,m,\delta,\sigma_{i})_{C}$ be the set of all linear maps $T\colon X_{F,m}\to V_{i}$  such that $\|T\|\leq C$ and
\[\|T(s_{1}\cdots s_{k}x_{j})-\sigma_{i}(s_{1})\cdots \sigma_{i}(s_{k})T(x_{j})\|<\delta\]
if $1\leq j,k\leq m,s_{1},\cdots,s_{k}\in F.$ If $C=1$ we shall use $\Hom_{\Gamma}(S,F,m,\delta,\sigma_{i})$ instead of $\Hom_{\Gamma}(S,F,m,\delta,\sigma_{i})_{1}.$}\end{definition}

	We shall frequently deal with inducing pseudonorms on $l^{\infty}(\NN,V)$ from pseudornoms on $l^{\infty}(\NN).$ For this, we use the following notation: if $\rho$ is a pseudonorm on $l^{\infty}(\NN)$ and $V$ is a Banach space, we let $\rho_{V}$ be the pseudonorm on $l^{\infty}(\NN,V)$ defined by $\rho_{V}(f)=\rho(j\mapsto \|f(j)\|).$
\begin{definition} \emph{Let $\Sigma,S$ be as in the proceeding definition and let $\rho$ be a pseudonorm on $l^{\infty}(\NN)$.  Let $\alpha_{S}\colon B(X_{F,m},V_{i})\to l^{\infty}(\NN,V_{i})$ be given by $\alpha_{S}(T)(j)=\chi_{\{k\leq m\}}(j)T(x_{j}).$ We let} 
\[\widehat{d}_{\varepsilon}(\Hom_{\Gamma}(S,F,m,\delta,\sigma_{i}),\rho)=d_{\varepsilon}(\alpha_{S}(\Hom_{\Gamma}(S,F,m,\delta,\sigma_{i})),\rho_{V_{i}})\]
\emph{define the} dimension of $S$ with respect to $\rho$ \emph{by}
\begin{align*}
f.\dim_{\Sigma}(S,F,m,\delta,\varepsilon,\rho)&=\limsup_{i\to \infty}\frac{1}{\dim V_{i}}\widehat{d}_{\varepsilon}(\Hom_{\Gamma}(S,F,m,\delta,\sigma_{i}),\rho),\\
f.\dim_{\Sigma}(S,\varepsilon,\rho)&=\limsup_{\substack{e\in F\subseteq \Gamma \mbox{\emph{ finite }}\\m\in \NN\\ \delta>0}}f.\dim_{\Sigma}(S,F,m,\delta,\varepsilon,\rho)\\
f.\dim_{\Sigma}(S,\rho)&=\sup_{\varepsilon>0}f.\dim_{\Sigma}(S,\varepsilon,\rho),
\end{align*}
\emph{where the pairs $(F,m,\delta)$ are ordered as follows $(F,m,\delta)\leq (F',m',\delta')$ if $F\subseteq F',m\leq m',\delta\geq \delta'.$} 

	\emph{We also use}
\begin{align*}
\underline{f.\dim}_{\Sigma}(S,F,m,\delta,\varepsilon,\rho)&=\liminf_{i\to \infty}\frac{1}{\dim V_{i}}\widehat{d}_{\varepsilon}(\Hom_{\Gamma}(S,F,m,\delta,\sigma_{i}),\rho),\\
\underline{f.\dim}_{\Sigma}(S,\varepsilon,\rho)&=\liminf_{\substack{e\in F\subseteq \Gamma \mbox{\emph{ finite }}\\ m\in \NN\\ \delta>0}}f.\dim_{\Sigma}(S,F,m,\delta,\varepsilon,\rho)\\
\underline{f.\dim}_{\Sigma}(S,\rho)&=\sup_{\varepsilon>0}f.\dim_{\Sigma}(S,\varepsilon,\rho).
\end{align*}

\end{definition}

In section $\ref{S:invariance}$ we will show that
\[f.\dim_{\Sigma}(S,\rho)=\sup_{\varepsilon>0}\liminf_{(F,m,\delta)}\limsup_{i\to \infty}\frac{1}{\dim V_{i}}\widehat{d}_{\varepsilon}(\Hom_{\Gamma}(S,F,m,\delta,\sigma_{i}),\rho),\]
\[\underline{f.\dim}_{\Sigma}(S,\rho)=\sup_{\varepsilon>0}\limsup_{(F,m,\delta)}\liminf_{i\to \infty}\frac{1}{\dim V_{i}}\widehat{d}_{\varepsilon}(\Hom_{\Gamma}(S,F,m,\delta,\sigma_{i}),\rho).\]

We introduce two other versions of dimension, which will be used to prove that the above notion of dimension does not depend on the generating sequence.

\begin{definition} \emph{Let $X$ be a separable Banach space, we say that $X$ has the} $C$-bounded approximation property \emph{if there is a sequence $\theta_{n}\colon X\to X$ of finite rank maps such that $\|\theta_{n}\|\leq C$ and}
\[\|\theta_{n}(x)-x\|\to 0,\mbox{\emph{ for all $x\in X.$}}\]
\emph{ We say that $X$ has the} bounded approximation property \emph{if it has the $C$-bounded approximation property for some $C>0.$}
\end{definition}

\begin{definition} \emph{Let $X$ be a separable Banach space with a uniformly bounded action of a countable discrete group $\Gamma.$ Let $q\colon Y\to X$ be a bounded linear surjective map, where $Y$ is a separable Banach space with the bounded approximation property. A} $q$-dynamical filtration \emph{is a pair $\F=((a_{sj})_{(s,j)\in \Gamma\times \NN},(Y_{E,l})_{e\in E\subseteq\Gamma \mbox{ finite},l\in \NN})$ where $a_{sj}\in Y,$ $Y_{E,l}\subseteq Y$ is a finite dimensional linear subspace such that}\end{definition}
\begin{list}{ \arabic{pcounter}:~}{\usecounter{pcounter}}
\item $\sup_{(s,j)}\|a_{sj}\|<\infty,$\\
\item $q(a_{sj})=sq(a_{ej}),$\\
\item $(q(a_{ej}))_{j=1}^{\infty}$ \emph{is dynamically generating},\\
\item $Y_{E,l}\subseteq Y_{E',l'}$ \emph{if} $E\subseteq E',l\leq l'$\\
\item $\ker(q)=\overline{\bigcup_{(E,l)}Y_{E,l}\cap \ker(q)},$\\
\item $Y_{E,l}=\Span\{a_{sj}:s\in E^{l},1\leq j\leq l\}+\ker(q)\cap Y_{E,l}.$
\end{list}

	Note that if $X$ has the bounded approximation property and $Y=X$ with $q$ the identity, then a dynamical filtration simply corresponds to a choice of a dynamically generating sequence. In general, if $S=(x_{j})_{j=1}^{\infty}$ is a dynamically generating sequence, then there is always a $q$-dynamical filtration $\F=((a_{sj})_{(s,j)\in \Gamma\times \NN},Y_{F,l})$ such that $q(a_{ej})=x_{j}.$ Simply choose $a_{sj}$ such that $\|a_{sj}\|\leq C\|x_{j}\|$ and $q(a_{sj})=sx_{j}$ for some $C>0.$ If $(y_{j})_{j=1}^{\infty}$ is a dense sequence in $\ker(q),$ we can set
\[Y_{E,l}=\Span\{a_{sj}:(s,j)\in E^{l}\times \{1,\cdots,l\}\}+\sum_{j=1}^{l}\CC y_{j}.\]
	
	We can always find a Banach space $Y$ with the bounded approximation property and a quotient map $q\colon Y\to X,$ in fact it is a standard exercise that we can choose $Y=l^{1}(\NN).$

\begin{definition} \emph{A} quotient dimension tuple \emph{ is a tuple $(Y,q,X,\Gamma,\sigma_{i}\colon \Gamma\to \Isom(V_{i}))$ where $(X,\Gamma,\sigma_{i})$ is a dimension triple, $Y$ is a separable Banach space with the bounded approximation property and $q\colon Y\to X$ is a bounded linear surjection. }
\end{definition}

\begin{definition}\emph{ Let $(Y,q,X,\Gamma,\sigma_{i}\colon \Gamma\to \Isom(V_{i}))$ be a quotient dimension triple, and  let $\F=((a_{sj})_{(s,j)\in \Gamma\times \NN},Y_{F,l})$ be a $q$-dynamical filtration. For $e\in F\subseteq \Gamma$ finite, $m\in \NN,\delta,C>0$ we let $\Hom_{\Gamma}(\F,F,m,\delta,\sigma_{i})_{C}$ be the set of all bounded linear maps $T\colon Y\to V_{i}$ such that $\|T\|\leq C$ and}
\[\|T(a_{s_{1}\cdots s_{k}j})-\sigma_{i}(s_{1})\cdots \sigma_{i}(s_{k})T(a_{ej})\|<\delta\]
\[\left\|T\big|_{\ker(q)\cap Y_{F,l}}\right\|<\delta.\]
\emph{As before, if $C=1$ we will  use $\Hom_{\Gamma}(\F,F,m,\delta,\sigma_{i})$ instead of $\Hom_{\Gamma}(\F,F,m,\delta,\sigma_{i})_{C}.$}
\end{definition}

	Again, in the case $X$ has the bounded approximation property, we are simply looking at almost equivariant maps from $\Gamma$ to $V_{i},$ and this is similar in spirit to the definition of topological entropy in \cite{KLi}. In the general case, note that genuine equivariant maps from $X$ to $V_{i}$ would correspond to maps on $Y$ which vanish on the kernel of $q,$ and so that 
\[T(a_{s_{1}\cdots s_{k}j})=\sigma_{i}(s_{1})\cdots \sigma_{i}(s_{k})T(a_{ej}),\]
 so we are still looking at almost equivariant maps on $X,$ in a certain sense.

\begin{definition}
\emph{Fix a pseudonorm $\rho$ on $l^{\infty}(\NN),$ let $(Y,q,X,\Gamma,\Sigma=(\sigma_{i}\colon \Gamma\to \Isom(V_{i})))$ be a quotient dimension tuple, and $\F$ a $q$-dynamical filtration. Let $\alpha_{\F}\colon B(Y,V_{i})\to l^{\infty}(\NN,V_{i})$ be given by $\alpha_{\F}(\phi)=(\phi(a_{ej}))_{j=1}^{\infty}$ we again use $\widehat{d}_{\varepsilon}(A,\rho)=d_{\varepsilon}(\alpha_{\F}(A),\rho_{V_{i}}).$   We define the dimension of $\F$ with respect to $\rho,\Sigma$ as follows:}\end{definition}
\begin{align*}
f.\dim_{\Sigma}(\F,F,m,\delta,\varepsilon,\rho)&=\limsup_{i\to \infty}\frac{1}{\dim V_{i}}\widehat{d}_{\varepsilon}(\Hom_{\Gamma}(\F,F,m,\delta,\sigma_{i}),\rho),\\
f.\dim_{\Sigma}(\F,\varepsilon,\rho)&=\inf_{\substack{e\in F\subseteq \Gamma \mbox{ finite } \\ m\in \NN \\ \delta>0}}f.\dim_{\Sigma}(\F,F,m,\delta,\varepsilon,\rho),\\
f.\dim_{\Sigma}(\F,\rho)&=\sup_{\varepsilon>0}f.\dim_{\Sigma}(\F,\varepsilon,\rho).
\end{align*}

Note that unlike $f.\dim_{\Sigma}(S,F,m,\delta,\varepsilon,\rho)$ we know that $f.\dim_{\Sigma}(\F,F,m,\delta,\varepsilon,\rho)$ is smaller when we enlarge $F$ and $m$ and shrink $\delta,$ thus the infimum is a limit and there are no issues between equality of limit supremums and limit infimums for this definition.

\begin{definition} \emph{Let $Y,X$ be Banach spaces, and let $\rho$ be a pseduonorm on $B(X,Y).$ For $\varespilon>0,0<M\leq \infty,$ and $A,C\subseteq B(X,Y),$ the set $C$ is said to} $(\varepsilon,M)$ contain $A$ \emph{if for every $T\in A,$ there is a $S\in C$ such that $\|S\|\leq M$ and $\rho(S-T)<\varepsilon.$ In this case we shall write $A\subseteq_{\varepsilon,M}C.$ We let $d_{\varepsilon,M}(A,\rho)$ be the smallest dimension of a linear subspace which $(\varepsilon,M)$ contains $A.$}\end{definition}

\begin{definition}\emph{ Let $(Y,q,X,\Gamma,\sigma_{i}\colon \Gamma\to \Isom(V_{i}))$ be a quotient dimension tuple.  Let $\F=(a_{sj},Y_{F,l})$ be a $q$-dynamical filtration. Fix a sequence of pseudonorms of $\rho_{i}$ on $B(Y,V_{i})$ and $0<M\leq \infty,$ set }\end{definition}
\begin{align*}
\opdim_{\Sigma,M}(\F,F,m,\delta,\varepsilon,\rho_{i})&=\limsup_{i\to \infty}\frac{1}{\dim V_{i}}d_{\varepsilon,M}(\Hom_{\Gamma}(\F,F,m,\delta,\sigma_{i}),\rho_{i}),\\
\opdim_{\Sigma,M}(\F,\varepsilon,\rho_{i})&=\inf_{\substack{e\in F\subseteq \Gamma \mbox{ finite }\\ m\in \NN\\ \delta>0}}\opdim_{\Sigma,M}(\F,F,m,\delta,\varepsilon,\rho),\\
\opdim_{\Sigma,M}(\F,\rho_{i})&=\sup_{\varepsilon>0}\opdim_{\Sigma,M}(\F,\varepsilon,\rho).
\end{align*}

As before, we shall use
\[\underline{\opdim}_{\Sigma,M}(\F,\rho_{i}),\underline{f.\dim}_{\Sigma}(\F,\rho)\]
for the same definitions as above, but replacing the limit supremum with the limit infimum.

	By scaling, 
\[\inf_{0<M<\infty}\opdim_{\Sigma,M}(\F,\rho_{i}),\opdim_{\Sigma,\infty}(\F,\rho_{i}),f.\dim_{\Sigma}(S,\rho),f.\dim_{\Sigma}(\F,\rho)\]
 remain the same when we replace $\Hom_{\Gamma}(\F,F,m,\delta,\sigma_{i}),$ ~$\Hom_{\Gamma}(S,F,m,\delta,\sigma_{i}),$~ by $\Hom_{\Gamma}(\F,F,m,\delta,\sigma_{i})_{C},$ $\Hom_{\Gamma}(S,F,m,\delta,\sigma_{i})_{C},$ for $C$ a fixed constant. This will be useful in several proofs.

	Note that if $\rho$ is a pseudonorm on $l^{\infty}(\NN),$ then we get a pseudonorm $\rho_{\F,i}$ on $B(Y,V_{i})$ by
\[\rho_{\F,i}(T)=\rho(j\mapsto\|T(a_{ej})\|).\]
 Further, for $0<M\leq \infty$
\[\opdim_{\Sigma,M}(\F,\rho_{\F,i})\geq f.\dim_{\Sigma}(\F,\rho).\]

\begin{definition} \emph{A}  product norm \emph{$\rho$ is a norm on $l^{\infty}(\NN)$ such that}\end{definition}
\begin{list}{\arabic{pcounter} : ~ }{\usecounter{pcounter}}
\item $\rho$  induces a topology stronger than the product topology,\\
\item $\rho$ induces a topology which agrees with the product topology on $\{f\in l^{\infty}(\NN):\|f\|_{\infty}\leq 1\}.$
\end{list}

	Typical examples are the $l^{p}$-norms:
\[\rho(f)^{p}=\sum_{j=1}^{\infty}\frac{1}{2^{j}}|f(j)|^{p}.\]

We shall show that there is constant $M>0,$ depending only on $Y,$ so that if  $\F,\F'$ are dynamical filtrations of $q$ and $S$ is a dynamically generating sequence, then for any two  product norms $\rho,\rho'$,
\[\opdim_{\Sigma,M}(\F,\rho'_{\F,i})=\opdim_{\Sigma,M}(\F,\rho_{\F,i})=f.\dim_{\Sigma}(\F,\rho)=\]
\[f.\dim_{\Sigma}(\F',\rho)=\dim_{\Sigma}(S,\rho).\]
and the same with $\dim$ replaced by $\underline{\dim}.$ In particular all these dimension only depend of the action of $\Gamma$ on $X,$ and give an isomorphism invariant. When we show all these equalities we let 
\[\dim_{\Sigma}(X,\Gamma)\]
denote any of these common numbers. 

	The equality between these dimensions is easier to understand in the case when $X$ has the bounded approximation property. When $X$ has the bounded approximation property, we can take $Y=X,q=\id$ and then the equality 
\[\opdim_{\Sigma,M}(\F,\rho_{\F,i})=f.\dim_{\Sigma}(S,\rho),\]
says the data of \emph{local} almost equivariant maps on $X$ is the same as the data of \emph{global} almost equivariant maps on $X.$ This is essentially because if we take $\theta_{E,l}\colon X\to X_{E,l}$ which tend pointwise to the identity, then any almost equivariant map on $X_{E,l}$ gives an almost equivariant map on $X$ by composing with $\theta_{E,l}.$ 

	Since the maps $\sigma_{i}\colon \Gamma\to \Isom(V_{i})$ are not assumed to have any structure, this invariant is uninteresting unless the maps $\sigma_{i}$ model the action of $\Gamma$ on $X$ in some manner. Thus we note that if $\Gamma$ is a sofic group, then the maps $\sigma_{i}\colon \Gamma\to S_{d_{i}}$ model at least the group $\Gamma$ in a reasonable manner. 

	Because $S_{n}$ acts naturally on $l^{p}(n)$ we get an induced sequence of maps $\sigma_{i}\colon \Gamma\to \Isom(l^{p}(d_{i}))$ and the above invariant measures how closely the action of $\Gamma$ on $X$ is modeled by these maps. When $\Gamma$ is sofic, and $\Sigma=(\sigma_{i}\colon \Gamma\to S_{d_{i}})$ is a sofic approximation and $\Sigma^{(p)}=(\sigma_{i}\colon \Gamma\to \Isom(l^{p}(d_{i})))$ are  the maps induced by the action of $S_{n}$ on $l^{p}(n),$ we let
\[\dim_{\Sigma,l^{p}}(X,\Gamma)=\dim_{\Sigma^{(p)}}(X,\Gamma)\]
\[\underline{\dim}_{\Sigma,l^{p}}(X,\Gamma)=\underline{\dim}_{\Sigma^{(p)}}(X,\Gamma).\]

	Similarly, if $\Gamma$ is $\R^{\omega}$-embeddable, and $\sigma_{i}\colon \Gamma\to U(d_{i})$ is a embedding sequence, then since $U(d_{i})$ is the isometry group of $l^{2}(d_{i})$ we shall let
\[\dim_{\Sigma,l^{2}}(X,\Gamma)=\dim_{\Sigma}(X,\Gamma)\]
\[\underline{\dim}_{\Sigma,l^{2}}(X,\Gamma)=\underline{\dim}_{\Sigma}(X,\Gamma).\]

	Just as $S_{n}$ acts on commutative $l^{p}$-Spaces, we have  two natural actions of $U(n)$ on non-commutative $L^{p}$-spaces. Let $S^{p}(n)$ be $M_{n}(\CC)$ with the norm
\[\|A\|_{S^{p}}=\Tr(|A|^{p})\]
where $|A|=(A^{*}A)^{1/2}.$ Then $U(n)$ acts isometrically on $S^{p}(n)$ by conjugation and by left multiplication. We shall use
\[\dim_{\Sigma,S^{p},\conjugate}(X,\Gamma)\]
for our dimension defined above, thinking of $\sigma_{i}$ as a map into $\Isom(S^{p}(n))$ by conjugation and
\[\dim_{\Sigma,S^{p},\multi}(X,\Gamma)\]
thinking of $\sigma_{i}$ as a map into $\Isom(S^{p}(n))$ by left multiplication. %

	One of our main applications will be showing that when $\Gamma$ is $\R^{\omega}$-embeddable
\[\underline{\dim}_{\Sigma,S^{p},\conjugate}(l^{p}(\Gamma)^{\oplus n},\Gamma)=\dim_{\Sigma,S^{p},\conjugate}(l^{p}(\Gamma)^{\oplus n},\Gamma)=n,\]
if $1\leq p\leq 2,$ and
\[\underline{\dim}_{\Sigma,l^{p}}(l^{p}(\Gamma)^{\oplus n},\Gamma)=\dim_{\Sigma,l^{p}}(l^{p}(\Gamma)^{\oplus n},\Gamma)=n,\]
if $1\leq p\leq 2,$ In particular the representations $l^{p}(\Gamma)^{\oplus n}$ are not isomorphic for different values of $n,$ if $\Gamma$ is  $\R^{\omega}$-embeddable.

\section{Invariance of the Definitions}\label{S:invariance}
	In this section we show that our various notions of dimension agree. Here is the main strategy of the proof. First we show that there is an $M>0,$ independent of $\F$ so that  
\[\opdim_{\Sigma,M}(\F,\rho_{\F,i})=f.\dim_{\Sigma}(\F,\rho),\]
 the constant $M$ comes from the constant in the definition of bounded approximation property.
 A compactness argument shows that
\[\opdim_{\Sigma,M}(\F,\rho_{\F,i})\]
does not depend on the choice of pseudonorm.  We then show that 
\[\opdim_{\Sigma,\infty}(\F,\rho_{\F,i})\]
does not depend on the choice of $\F,$ this is easier than trying to show that
\[f.\dim_{\Sigma}(S,\rho)\]
does not depend on the choice of $S.$ This is because the  maps used to define
\[\opdim_{\Sigma,\infty}(\F,\rho_{\F,i})\]
all have the same domain, which makes it easy to switch from one generating set to another, since we can use that generators for $\F$ have to be close to linear combinations of generators for $\F'.$ Then we show that
\[f.\dim_{\Sigma}(\F,\rho)=f.\dim_{\Sigma}(S,\rho),\]
this will reduce to showing that if we are given an almost equivariant map $\phi \colon Y\to V_{i}$ which is small on the kernel of $q,$ then there is a $T\colon X'\to V$ with $X'\subseteq X$ finite dimensional such that $T\circ q$ is close to $\phi$ on a prescribed finite set. 

	First we need a simple fact about spaces with the bounded approximation property.
\begin{proposition}\label{P:control} Let $Y$ be a separable Banach space with the $C$-bounded approximation property, and let $I$ be a countable directed set. Let $(Y_{\alpha})_{\alpha\in I}$ be an increasing net of subspaces of $Y$ such that
\[Y=\overline{\bigcup_{\alpha}Y_{\alpha}}.\]
Then  there are finite-rank maps $\theta_{\alpha}\colon Y\to Y_{\alpha}$  such that $\|\theta_{\alpha}\|\leq C$ and
\[\lim_{\alpha}\|\theta_{\alpha}(y)-y\|= 0\]
for all $y\in Y.$

\end{proposition}

\begin{proof}  Fix $y_{1},\cdots,y_{k}\in Y$ and $\varepsilon>0.$ Then there is a finite rank $\theta\colon Y\to Y$ such that
\[\|\theta(y_{j})-y_{j}\|<\varepsilon,\]
\[\|\theta\|\leq C.\]
Write
\[\theta=\sum_{j=1}^{n}\phi_{j}\otimes x_{j}\]
with $\phi_{j}\in Y^{*}$ and $x_{j}\in Y.$ If $\alpha$ is sufficiently large, then we can find $x_{j}'\in Y_{\alpha}$ close enough to $x_{j}$ so that if we let 
\[\theta_{0}=\sum_{j=1}^{n}\phi_{j}\otimes x_{j}',\]
\[\widetilde{\theta}=\begin{cases}
\theta_{0}&\textnormal{ if $\|\theta_{0}\|\leq C$}\\
C\frac{\theta_{0}}{\|\theta_{0}\|}&\textnormal{otherwise}
\end{cases}\]
then
\[\|\widetilde{\theta}(y_{j})-y_{j}\|<2\varepsilon.\]

	Now let $(y_{j})_{j=1}^{\infty}$ be a dense sequence in $Y,$ and let 
\[\alpha_{1}\leq \alpha_{2}\leq \alpha_{3}\leq \cdots \]
 with $\alpha_{j}\in I$ be such that for all $\beta\in I,$ there is a $j$ such that $\beta\leq \alpha_{j}.$ By the preceding paragraph, we can inductively construct an increasing sequence $n_{k}$ of integers and finite-rank maps
\[\theta_{k}\colon Y\to Y_{\alpha_{n_{k}}}\]
such that
\[\|\theta_{k}\|\leq C\]
\[\|\theta_{k}(y_{j})-y_{j}\|\leq 2^{-k}\mbox{ if $j\leq k$}.\]
Set $\theta_{\alpha}=\theta_{\alpha_{n_{k}}}$ if $k$ is the largest integer such that $\alpha_{n_{k}}$ is not bigger than $\alpha.$ Let $\theta_{\alpha}=0$ if $\alpha<\alpha_{1}.$  Then $\theta_{\alpha}$ has the desired properties.
\end{proof}

\begin{lemma}\label{L:linearmap} Let $(Y,q,X,\Gamma,\Sigma=(\sigma_{i}\colon \Gamma\to \Isom(V_{i})))$ be a quotient dimension tuple. Let $\F=((a_{sj})_{(s,j)\in \Gamma\times \NN},Y_{F,l})$ be a $q$-dynamical filtration and $\rho$ a product norm, and let $C>0$ be such that $Y$ has the $C$-bounded approximation property. Fix $M>C.$
	Then for any $V\subseteq Y$ finite-dimensional, and $\kappa>0,$ there is a $F\subseteq \Gamma$ finite $m\in \NN,$ $\delta,\varepsilon>0$ and linear maps
\[L_{i}\colon l^{\infty}(\NN,V_{i})\to B(Y,V_{i})\]
so that if $\phi\in \Hom_{\Gamma}(\F,F,m,\delta,\sigma_{i}),f\in l^{\infty}(\NN,V_{i})$ satisfy $\rho_{V_{i}}(\alpha_{\F}(\phi)-f)<\varepsilon,$ then
\[\|L_{i}(f)\|\leq M,\]
\[\left\|L_{i}(f)\big|_{V}-\phi\big|_{V}\right\|<\kappa.\]
\end{lemma}

\begin{proof}  Note that for every $V$ finite-dimensional there are a $E\subseteq \Gamma$ finite, $l\in \NN,$ such that 
\[\max_{\substack{v\in V \\ \|v\|=1}}\inf_{\substack{ w\in Y_{E,l}\\ \|w\|=1}}\|v-w\|<\kappa,\]
so we may assume that $V=Y_{E,l}$ for some $E,l.$

	Fix $\eta>0$ to be determined later. By the preceding proposition let  $\theta_{F,k}\colon Y\to Y_{F,k}$ be such that
\[\|\theta_{F,k}\|\leq C,\]
\[\lim_{(F,k)}\|\theta_{F,k}(y)-y\|=0\mbox{ for all $y\in Y$}.\]
Choose $F,m$ sufficiently large such that
\[\|\theta_{F,m}\big|_{Y_{E,l}}-\id \big|_{Y_{E,l}}\|\leq \eta.\]

	Let $\B_{F,m}\subseteq F^{m}\times \{1,\cdots,m\}$ be such that $\{q(a_{sj}):(s,j)\in \B_{F,m}\}$ is a basis for $X_{F,m}=\Span\{q(a_{sj}):(s,j)\in F^{m}\times \{1,\cdots,m\}\}.$ Define
\[\widetilde{L_{i}}\colon l^{\infty}(\NN,V_{i})\to B(X_{F,m},V_{i})\]
by
\[\widetidle{L_{i}}(f)(q(a_{sj}))=\sigma_{i}(s)f(j)\mbox{ for $(s,j)\in \B_{F,m}$}.\]
 We claim that if $\delta>0,\varepsilon'>0$ are sufficently small, $\phi\in \Hom_{\Gamma}(\F,F^{m},m,\delta,\sigma_{i})$ and $f\in l^{\infty}(\NN,V_{i})$ satisfy
\[\rho_{V_{i}}(f-\alpha_{\F}(\phi))<\varepsilon',\]
then
\begin{equation}\label{E:normestimate}
\|\widetilde{L_{i}}(f)\circ q\big|_{Y_{F,m}}-\phi\big|_{Y_{F,m}}\|\leq \eta.
\end{equation}

	By finite-dimensionality, there is a $D(F,m)>0$ such that if $v\in \ker(q)\cap Y_{F,m},(d_{tr})\in \CC^{\B_{F,m}},$ then
\[\sup(\|v\|,|d_{tr}|)\leq D(F,m)\left\|v+\sum_{(t,r)\in \B_{F,m}}d_{tr}a_{tr}\right\|.\]

	Thus if $x=v+\sum_{(t,r)\in \B_{F,m}}d_{tr}a_{tr}$ with $v\in \ker(q)\cap Y_{F,m}$ has $\|x\|=1,$ then
\begin{align*}
\|\widetilde{L_{i}}(f)(q(x))-\phi(x)\|&\leq D(F,m)\delta+D(F,m)\sum_{(t,r)\in \B_{F,m}}\|\phi(a_{tr})-\sigma_{i}(t)f(r)\|\\
&\leq D(F,m)\delta+D(F,m)|F|^{m}m\delta+\sum_{(t,r)\in \B_{F,m}}\|\phi(a_{er})-f(r)\|,
\end{align*}
if $\delta<\frac{\eta}{2D(F,m)(1+|F|^{m}m)},$ and $\varepsilon'>0$ is small enough so that $\rho(g)<\varepsilon'$ implies
\[\sum_{(t,r)\in \B_{F,m}}|g(r)|<\frac{\eta}{2},\]
then our claim holds.

	So assume that $\delta,\varepsilon'>0$ are small enough so that $(\ref{E:normestimate})$ holds, and set $L_{i}(f)=\widetilde{L_{i}}(f)\circ q\big|_{Y_{F,m}}\circ \theta_{F,m}.$ Then 
\[\|L_{i}(f)\|\leq C(1+\eta)\]
and for  $\phi,f$ as above and $y\in Y_{E,l}$
\[\|L_{i}(f)(y)-\phi(y)\|\leq (1+\eta)\|\theta_{F,m}(y)-y\|+\|\widetilde{L_{i}}(f)\circ q(y)-\phi(y)\|\leq\ (2+\eta)\eta\|y\|.\]
So we force $\eta$ to be small enough so that $(2+\eta)\eta<\kappa,C(1+\eta)<M.$

\end{proof}

\begin{lemma}\label{L:differentdim} Let $(Y,q,X,\Gamma,\Sigma=(\sigma_{i}\colon \Gamma\to \Isom(V_{i})))$ be a quotient dimension tuple. 

	Let $\F=((a_{sj})_{(s,j)\in \Gamma\times \NN},Y_{E,l})$ be a $q$-dynamical filtration, and $\rho$ a  product norm, suppose that $Y$ has the $C$-bounded approximation property.

(a) If $\infty\geq M>C,$ then
\[f.\dim_{\Sigma}(\F,\rho)=\opdim_{\Sigma,M}(\F,\rho),\]
\[\underline{f.\dim}_{\Sigma}(\F,\rho)=\underline{\opdim}_{\Sigma,M}(\F,\rho).\]

(b) If $\rho'$ is another  product norm then for all $0<M<\infty,$
\[\opdim_{\Sigma,M}(\F,\rho_{\F,i})=\opdim_{\Sigma,M}(\F,\rho'_{\F,i}),\]
\[\underline{\opdim}_{\Sigma,M}(\F,\rho_{\F,i})=\underline{\opdim}_{\Sigma,M}(\F,\rho'_{\F,i}).\]

\end{lemma}

\begin{proof}  

 (a)  First note that 
\[\opdim_{\Sigma,M}(\F,\rho)\geq \opdim_{\Sigma,\infty}(\F,\rho)\geq f.\dim_{\Sigma}(\F,\rho)\]
so it suffices to handle the case that $M<\infty.$

Let $A>0$ be such that
\[\|a_{sj}\|\leq A\mbox{ for all $(s,j)\in \Gamma\times \NN$}\]
Take $1>\varepsilon>0.$ Let $k$ be such that if $f\in l^{\infty}(\NN),$ and $\|f\|_{\infty}\leq 1,$ and $f$ is supported on $\{n:n\geq k\},$ then $\rho(f)<\varepsilon.$  Since $\rho$ induces a topology weaker than the norm topology, we can find an $\varepsilon>\kappa>0$ such that 
\[\rho(f)<\varepsilon\]
if
\[\|f\|_{\infty}\leq \kappa.\]

By Lemma $\ref{L:linearmap},$ let $e\in F\subseteq \Gamma$ be finite, $m\in \NN,$ $\varepsilon>\varepsilon'>0,\kappa>\delta>0$ and $L_{i}\colon l^{\infty}(\NN,V_{i})\to B(Y,V_{i})$ be such that if $\phi\in \Hom_{\Gamma}(\F,F,m,\delta,\sigma_{i})$ and $f\in l^{\infty}(\NN,V_{i})$ has $\rho_{V_{i}}(\alpha_{\F}(\phi)-f)<\varepsilon',$ then
\[\big\|L_{i}(f)\big|_{Y_{\{e\},k}}-\phi\big|_{Y_{\{e\},k}}\|<\kappa,\]
\[\|L_{i}(f)\|\leq M.\]

	Then if $\phi,f$ are as above we have
\[\rho_{\F,i}(\phi-L_{i}(f))\leq (M+1)A\varepsilon+\rho(\chi_{l\leq k}(j)(\|\phi(a_{ej})-L_{i}(f)(a_{ej})\|)_{j=1}^{\infty})\]
and for $j\leq k$ 

\[\|\phi(a_{ej})-L_{i}(f)(a_{ej})\|\leq A(M+1)\kappa.\]

Thus

\[\rho_{\F,i}(\phi-L_{i}(f))\leq (M+1)(A+1)\varepsilon.\]
This implies that
\[d_{((M+1)(A+1)\varepsilon,M}(\Hom_{\Gamma}(\F,F',m',\delta',\sigma_{i}),\rho_{\F,i})\leq d_{\varepsilon'}^{\widehat{}}(\Hom_{\Gamma}(\F,F',m',\delta',\sigma_{i}),\rho_{\F,i})\]
for all $F'\supseteq F,m'\geq m,$ and all  $\delta'<\delta.$ This completes the proof.

(b) This is a simple consequence of the compactness of the $\|\cdot\|_{\infty}$ unit ball of $l^{\infty}(\NN)$ in the product topology.

\end{proof}

\begin{lemma}\label{L:samepseudo} Let $(Y,q,X,\Gamma,\sigma_{i}\colon \Gamma\to \Isom(V_{i}))$ be a quotient dimension tuple. Let $\F,\F'$ be two  $q$-dynamical filtrations.  If $\rho_{i}$ is any fixed sequence of pseudonorms on $B(Y,V_{i}),$ then for all $0< M\leq \infty,$
\[ \opdim_{\Sigma,M}(\F,\rho_{i})=\opdim_{\Sigma,M}(\F',\rho_{i}),\]
\[ \underline{\opdim}_{\Sigma,M}(\F,\rho_{i})=\underline{\opdim}_{\Sigma,M}(\F',\rho_{i}),\]
\end{lemma}
\begin{proof} Let $\F'=((a'_{sj})_{(s,j)\in \Gamma\times \NN},Y'_{E,l}),$ $\F=((a_{sj})_{(s,j)\in \Gamma\times \NN},Y_{E,l}).$ We do the proof for $\opdim_{\Sigma},$ the other case is  proved in the same manner. Let $C>0$ be such that $\|sx\|\leq C\|x\|$ for all $s\in \Gamma,x\in X$ and such that $\|a_{sj}\|,\|a'_{sj}\|\leq C.$ Fix $F\subseteq \Gamma$ finite, and $m\in \NN,\delta>0.$ Fix $\eta>0$ which will depend upon $F,m,\delta$ in a manner to be determined later. 

	Choose $E\subseteq \Gamma$ finite $l\in \NN,$ such that for $1\leq j\leq m,s\in F^{m}$ there are $c_{j,t,k}$ with $(t,k)\in E\times \{1,\cdots,l\}$ and $v_{sj}\in Y_{E,l}'\cap \ker(q)$ such that
\[\left\|a_{sj}-v_{sj}-\sum_{(t,k)\in E\times \{1,\cdots,l\}}c_{j,t,k}a'_{stk}\right\|<\eta,\]
and so that for every $w\in Y_{F,m}\cap \ker(q)$ there is a $v\in Y_{E,l}'\cap \ker(q)$ such that $\|v-w\|\leq \eta\|w\|.$ Let $A(\eta)=\sup(|c_{j,t,k}|,\sup \|v_{sj}\|)$

	Set $m'=2\max(m,l)+1,F'=[(F\cup F^{-1}\cup \{e\})(E\cup E^{-1}\cup \{e\})]^{2 m'+1},$ we claim that we can choose $\delta'>0,\eta>0$ small so that 
\[\Hom_{\Gamma}(\F',F',m',\delta',\sigma_{i})\subseteq \Hom_{\Gamma}(\F,F,m,\delta,\sigma_{i}).\]

	If $T\in\Hom_{\Gamma}(\F',F',m',\delta',\sigma_{i}),$ $1\leq j,r\leq m,$ and $s_{1},\cdots,s_{r}\in F$ then
\[\|T(a_{s_{1}\cdots s_{r}j})-\sigma_{i}(s_{1})\cdots \sigma_{i}(s_{r})T(a_{ej})\|\leq\]
\[2\eta+\|T(v_{sj})\|+\|\sigma_{i}(s_{1})\cdots \sigma_{i}(s_{r})T(v_{ej})\|+\]
\[\left\|\sum_{(t,k)\in E\times \{1,\cdots,l\}}c_{j,t,k}[T(a'_{s_{1}\cdots s_{r}tk})-\sigma_{i}(s_{1})\cdots \sigma_{i}(s_{r})T(a'_{tk})]\right\|\leq\]
\[2\eta+\delta'A(\eta)+\delta'A(\eta)+2|E|lA(\eta)\delta'.\]
By choosing $\eta<\delta/2,$ and then choosing $\delta'$ very small we can make the above expression less than $\delta.$ If we also force $\delta'<\delta/2$ our choice of $\eta$ implies that 
\[\|T(w)\|\leq \delta\|w\|\]
for $T$ as above and $w\in Y_{F,m}\cap \ker(q).$ This completes the proof. 
\end{proof}

Because of the above lemma, the only difficulty in proving that $\opdim_{\Sigma}(\F,\rho_{\F,i})$ does not depend on the choice of $\F$ is switching the pseudonorm from $\rho_{\F,i}$ to $\rho_{\F',i}.$ Because of this we will investigate how the dimension changes when we switch pseudonorms. 

\begin{definition} \emph{ Let $(Y,q,X,\Gamma,\Sigma=(\sigma_{i}\colon \Gamma\to \Isom(V_{i})))$ be a quotient dimension tuple, and fix a $q$-dynamical filtration $\F.$ If $\rho_{i},q_{i}$ are pseduornoms on $B(Y,V_{i})$ we say that $\rho_{i}$ is} $(\F,\Sigma)$-weaker than $q_{i}$ \emph{ and write $\rho_{i}\preceq_{\F,\Sigma}q_{i}$ if the following holds. For every $\varepsilon>0,$  there are $F\subseteq \Gamma$ finite, $\delta,\varepsilon'>0,$ $m,i_{0}\in \NN,$ and linear maps $L_{i}\colon B(Y,V_{i})\to B(Y,V_{i})$ for $i\geq i_{0}$ such that if $\phi\in \Hom_{\Gamma}(\F,F,m,\delta,\sigma_{i})$ and $\psi\in B(Y,V_{i})$ satisfy $q_{i}(\phi-\psi)<\varepsilon',$ then $\rho_{i}(\phi-L_{i}(\psi))<\varepsilon.$ We say that} $\rho_{i}$ is $(\F,\Sigma)$ equivalent to $q_{i},$ \emph{ and write $\rho_{i}\thicksim_{\F,\Sigma}q_{i},$ if $\rho_{i}\preceq_{\F,\Sigma}q_{i}$ and $q_{i}\preceq_{\F,\Sigma}\rho_{i}.$ }\end{definition}

\begin{lemma}\label{L:Fweaker} Let $(Y,X,q,\Gamma,\Sigma)$ be a quotient dimension tuple and $\F$ a $q$-dynamical filtration.

(a) If $\rho_{i},q_{i}$ are pseudonorms with $\rho_{i}\preceq_{\F,\Sigma}q_{i},$ then
\[\opdim_{\Sigma,\infty}(\F,\rho_{i})\leq \opdim_{\Sigma,\infty}(\F,q_{i}),\]
\[\underline{\opdim}_{\Sigma,\infty}(\F,\rho_{i})\leq \underline{\opdim}_{\Sigma,\infty}(\F,q_{i}).\]

(b)  Let $\F'=((a'_{sj})_{(s,j)\in \Gamma\times \NN},Y'_{E,l}),$ $\F=((a_{sj})_{(s,j)\in \Gamma\times \NN},Y_{E,l})$ be  $q$-dynamical filtrations. Let $\rho$ be any  product norm. Define a pseudonorm on $B(Y,V_{i})$ by $\rho_{\F,i}(\phi)=\rho((\|\phi(a_{ej})\|)_{j=1}^{\infty}),$ and similarly define $\rho_{\F',i}.$ Then

 \[\rho_{\F',i}\preceq_{\F,\Sigma}\rho_{\F,i}.\]

\end{lemma}

\begin{proof}
Let $\Sigma=(\sigma_{i}\colon \Gamma\to \Isom(V_{i})).$

(a) This follows directly follow the definitions.

(b) Let $C>0$ be such that $Y$ has the $C$-bounded approximation property and
\[\|a_{sj}\|\leq C\]
\[\|a'_{sj}\|\leq C\]

Choose $m\in \NN$ such that $\rho(f)<\varepsilon$ if $\|f\|_{\infty}\leq 1$ and $f$ is supported on $\{n:n\geq m\},$ and let $\kappa>0$ be such that $\rho(f)<\varepsilon$ if $\|f\|_{\infty}\leq \kappa.$

	By Lemma  $\ref{L:linearmap}$ choose $F'\supseteq F$ finite $m\leq m'\in \NN,$ and $\delta,\varepsilon>0$ and 
\[\widetilde{L_{i}}\colon l^{\infty}(\NN,V_{i})\to B(Y,V_{i})\]
so that if $f\in l^{\infty}(\NN,V_{i})$ and $\phi \in \Hom_{\Gamma}(\F,F',m',\delta,\sigma_{i})$ has $\rho_{V_{i}}(\alpha_{\F}(\phi)-f)<\varepsilon'$ then
\[\left\|\widetilde{L_{i}}(f)\big|_{Y_{\{e\},m}'}-\phi\big|_{Y_{\{e\},m}'}\right\|<\kappa,\]
\[\|\widetilde{L_{i}}(f)\|\leq 2C.\]
Let $L_{i}\colon B(Y,V_{i})\to B(Y,V_{i})$ be given by $L_{i}(\psi)=\widetilde{L_{i}}(\alpha_{\F}(\psi)).$

	Suppose $\phi\in \Hom_{\Gamma}(\F,F',m',\delta',\sigma_{i})$ and $\psi\in B(Y,V_{i})$ satisfy $\rho_{\F,i}(\phi -\psi)<\varepsilon'.$ Then, for $1\leq j\leq m$ we have

\[\|\phi(a'_{ej})-L_{i}(\psi)(a'_{ej})\|\leq C\kappa.\]
Our choice of $m,\kappa$ then imply that $\rho_{\F',i}(\phi-L_{i}(\psi))<2C(C+1)\varepsilon.$ This completes the proof.

\end{proof}

\begin{cor}\label{C:ind} Let $(Y,q,X,\Gamma,\sigma_{i}\colon \Gamma\to \Isom(V_{i}))$ be a quotient dimension tuple. Let $\rho,\rho'$ be two  product norms. For any two $q$-dynamical filtrations $\F,\F'$ we have
\[\opdim_{\Sigma,\infty}(\F,\rho_{\F,i})=\opdim_{\Sigma,\infty}(\F',\rho_{\F',i})=\opdim_{\Sigma,\infty}(\F',\rho'_{\F',i}).\]
\[\underline{\opdim}_{\Sigma,\infty}(\F,\rho_{\F,i})=\underline{\opdim}_{\Sigma,\infty}(\F',\rho_{\F',i})=\underline{\opdim}_{\Sigma}(\F',\rho'_{\F',i}).\]
\end{cor}

\begin{proof} Combining Lemmas $\ref{L:differentdim},\ref{L:Fweaker},$ and $\ref{L:samepseudo}$  we have 
\[\opdim_{\Sigma,\infty}(\F',\rho'_{\F',i})=\opdim_{\Sigma,\infty}(\F',\rho_{\F',i})\leq\opdim_{\Sigma,\infty}(\F,\rho_{\F,i}).\]
The opposite inequality follows by symmetry.

\end{proof}

Because of the preceding corollary $f.\dim_{\Sigma}(\F,\rho)$ only depends on the action of $\Gamma$ and the quotient map $q\colon Y\to X.$ Thus we can define
\[\dim_{\Sigma}(q,\Gamma)=\opdim_{\Sigma,\infty}(\F,\rho_{\F,i})=f.\dim_{\Sigma}(\F,\rho)\]
where $\F$ is any $q$-dynamical filtration  and $\rho$ is any  product norm. 

	We now proceed to show that $\dim_{\Sigma,\infty}(q,\Gamma)$ does not depend on $q,$ as stated before the idea is to prove that
\[\dim_{\Sigma}(q,\Gamma)=f.\dim_{\Sigma}(S,\rho)\]
where $S$ is any dynamically generating sequence for $X.$ 
	
	For this, we will prove that we can approximate maps $T$ on $Y$ which almost vanish on the kernel of $q,$ by maps on $X.$ For the proof, we need the construction of ultraproducts of Banach spaces. 

	Let $X_{n}$ be a sequence of Banach spaces and $\omega\in \beta\NN\setminus\NN$ a free ultrafilter. We define the ultraproduct of the $X_{n},$ written $\prod^{\omega}X_{n}$ by
\[\prod^{\omega}X_{n}=\{(x_{n})_{n=1}^{\infty}:x_{n}\in X_{n},\sup_{n}\|x_{n}\|<\infty\}/\{(x_{n})_{n=1}^{\infty}:x_{n}\in X_{n},\lim_{n\to \omega}\|x_{n}\|=0\}.\]
We use $(x_{n})_{n\to \omega}$ for the image of $(x_{n})_{n=1}^{\infty}$ under the canonical quotient map to
\[\prod^{\omega}X_{n}.\]
	If a set $A\subseteq \NN$ is in $\omega,$ we will say that $A$ is $\omega$-large.

\begin{lemma}\label{L:section} Let $X,Y$ be Banach spaces with $X$  and $q\colon Y\to X$ a bounded linear surjective map. Let $F\subseteq X$ be finite and $Z$ a finite-dimensional subspace of $Y$ with $q(F)\subseteq Z.$ Let $C>0$ be such that for all $x\in X,$ there is a $y\in Y$ with $\|y\|\leq C\|x\|$ such that $q(y)=x,$ and fix $A>C.$ Let $I$ be a countable directed set, and $(Y_{\alpha})_{\alpha \in I}$ a net of subspaces of $Y$ such that $Y_{\alpha}\subseteq Y_{\beta}$ if $\alpha\leq \beta,$ and
\[q(Y_{\alpha})\supseteq Z,\]
\[\ker(q)=\overline{\bigcup_{\alpha}Y_{\alpha}\cap \ker(q)},\]
\[F\subseteq \bigcup_{\alpha}Y_{\alpha}.\]
Then for all $\varepsilon>0,$ there are a $\delta>0$ and $\alpha_{0}$ with the following property. If $\alpha\geq \alpha_{0}$ and  $W$ is a Banach space with $T\colon Y_{\alpha}\to W$ a linear contraction such that
\[\left\|T\big|_{\ker(q)\cap Y_{\alpha}}\right\|\leq \delta,\]
then there is a $S\colon Z\to W$ such that $\|S\|\leq A$ and
\[\|T(x)-S\circ q(x)\|\leq \varepsilon,\]
for all $x\in F.$
\end{lemma}
\begin{proof} Note that our assumptions imply
\[Y=\overline{\bigcup_{\alpha}Y_{\alpha}}.\]

Fix a countable increasing sequence $\alpha_{n}$ in $I,$ such that for every $\beta\in I$ there is an $n$ such that $\beta\leq \alpha_{n}.$ Assume also that $F\subseteq Y_{\alpha_{1}}.$ Since $I$ is directed, if the claim is false, then we can find an $\varepsilon>0$ and an increasing sequence $\beta_{n}$ with $\beta_{n}\geq \alpha_{n}$ and a $T_{n}\colon Y_{\beta_{n}}\to W_{n}$ such that $\|T_{n}\|\leq 1,$
\[\left\|T_{n}\big|_{\ker(q)\cap Y_{\beta_{n}}}\right\|\leq 2^{-n},\]
 and for every $S\colon X\to W_{n}$ with $\|S\|\leq A,$
\[\|T_{n}(x)-S\circ q(x)\|\geq \varepsilon,\mbox{ for some $x\in F$}.\]

	Fix $\omega\in \beta\NN\setminus \NN$ and let 
\[W=\prod^{\omega}W_{n}.\]
Define 
\[T\colon \bigcup_{n}Y_{\beta_{n}}\to W\]
by
\[T(x)=(T_{n}(x))_{n\to \omega},\]
note that for any $k,$ the map $T_{n}$ is defined on $Y_{\beta_{k}}$ for $n\geq k,$ so $T$ is well-defined. Also 
\[\|T(x)\|\leq \|x\|\]
\[T(x)=0\mbox{ on $\bigcup_{n}Y_{\beta_{n}}\cap \ker(q)$}.\]

	Our density assumptions imply that $T$ extends uniquely to a bounded linear map, still denoted $T,$ from $Y$ to $W,$ which vanishes on the kernel of $q.$ Thus there is $S\colon Z\to W$ such that $T=S\circ q,$ and our hypothesis on $C$ implies that $\|S\|\leq C.$ 

	Since $Z$ is finite dimensional, we can find $S_{n}\colon X\to W_{n}$ such that $S(x)=(S_{n}(x))_{n\to \omega}.$ Compactness of the unit sphere of $Z$ and a simple diagonal argument show that 
\[C\geq \|S\|=\lim_{n\to \omega}\|S_{n}\|.\]
Thus $B=\{n:\|S_{n}\|<A\}$ is an $\omega$-large set, and by hypothesis 
\[B=\bigcup_{x\in F}\{n\in B:\|T_{n}(x)-S_{n}(q(x))\|\geq \varepsilon\}.\]
Since $B$ is $\omega$-large, there is some $x\in F$ such that
\[\{n\in B:\|T_{n}(x)-S_{n}(q(x))\|\geq \varepsilon\}\]
is $\omega$-large.
But then  $T(x)\ne S\circ q(x),$ a contradiction.

\end{proof}

\begin{lemma}\label{L:quotientgenerating} Let $(Y,q,X,\Gamma,\Sigma=(\sigma_{i}\colon \Gamma\to \Isom(V_{i})))$ be a quotient dimension tuple. Fix a dynamically generating sequence $S$ in $X,$ and $\rho$ a   product norm. Then
\[\dim_{\Sigma}(q,\Gamma)=f.\dim_{\Sigma}(S,\rho).\]
\[\underline{\dim}_{\Sigma}(q,\Gamma)=\underline{f.\dim}_{\Sigma}(S,\rho).\]
\end{lemma}

\begin{proof} We will only do the proof for $\dim.$

	Let $S=(x_{j})_{j=1}^{\infty}$ and let $\F=((a_{sj})_{(s,j)\in \Gamma\times \NN},Y_{E,l})$ be a dynamical filtration such that $q(a_{ej})=x_{j}.$ Let $C>0$ be such that
\[\sup_{(s,j)}\|a_{sj}\|\leq C\]
\[\sup_{j}\|x_{j}\|\leq C\]
\[\|q\|\leq C,\]
\[\mbox{ for every $x\in X,$ there is a $y\in Y$ such that $q(y)=x$ and $\|y\|\leq C\|x\|$},\]
and so that $Y$ has the $C$-bounded approximation property. By Proposition, \ref{P:control}, we may find $\theta_{E,l}\colon Y\to Y_{E,l}$  such that $\|\theta_{E,l}\|\leq C$ and
\[\lim_{(E,l)}\|\theta_{E,l}(y)-y\|=0\mbox{ for all $y\in Y$}.\]

	We first show that 
\[\dim_{\Sigma}(q,\Gamma)\geq f.\dim_{\Sigma}(S,\rho).\]
For this, fix $\varepsilon>0,$ and choose $r\in \NN$ such that
\[\rho(f)<\varepsilon,\mbox{ if $f$ is supported on $\{n:n\geq r\}$ and $\|f\|_{\infty}\leq 1$},\]
as before choose $\varepsilon\geq \kappa>0$ such that if $\|f\|_{\infty}\leq \kappa,$ then
\[\rho(f)<\varepsilon.\]
Let $e\in E\subseteq \Gamma$ finite and $l\in \NN$ be such that if $E\subseteq F\subseteq \Gamma$ is finite, and $k\geq l$ then
\[\|\theta_{F,k}(a_{ej})-a_{ej}\|<\kappa\]
for $1\leq j\leq r.$ 

	Now fix $E\subseteq F\subseteq \Gamma$ finite, $l\leq m\in \NN,\delta>0.$ We claim that we can find $F\subseteq F'\subseteq \Gamma$ finite $m\leq m'$ in $\NN,$ $\delta>\delta'>0$ such that 
\[\Hom_{\Gamma}(S,F',m',\delta',\sigma_{i})\circ q\big|_{Y_{F',m'}}\circ \theta_{F',m'}\subseteq \Hom_{\Gamma}(\F,F,m,\delta,\sigma_{i})_{C^{2}}.\]

	 For $T\in \Hom_{\Gamma}(S,F',m',\delta',\sigma_{i}),$ for $1\leq j,k\leq m$ and $s_{1},\cdots, s_{k}\in F,$
\begin{align*}
&\|T\circ q\circ \theta_{F',m'}(a_{s_{1}\cdots s_{k}j})-\sigma_{i}(s_{1})\cdots \sigma_{i}(s_{k})T\circ q\circ \theta_{F',m'}(a_{ej})\| \\
&\leq C\|\theta_{F',m'}(a_{s_{1}\cdots s_{k}j})-a_{s_{1}\cdots s_{k}j}\|+C\|\theta_{F',m'}(a_{ej})-a_{ej}\| \\
&+\|T(s_{1}\cdots s_{k}x_{j})-\sigma_{i}(s_{1})\cdots \sigma_{i}(s_{k})T(x_{j})\| \\
&<C\|\theta_{F',m'}(a_{s_{1}\cdots s_{k}j})-a_{s_{1}\cdots s_{k}j}\|+C\|\theta_{F',m'}(a_{ej})-a_{ej}\|\\
&+\delta'.
\end{align*}

	Also for $y\in \ker(q)\cap Y_{F,m}$ we have
\[\|T\circ q\circ \theta_{F',m'}(y)\|\leq C\|\theta_{F',m'}(y)-y\|.\]

So it suffices to choose $\delta'<\min(\delta,\kappa)$  and then $F'\supseteq F,m'\geq \max(m,l,r)$ such that
\[C\|\theta_{F',m'}(a_{s_{1}\cdots s_{k}j})-a_{s_{1}\cdots s_{k}j}\|+C\|\theta_{F',m'}(a_{ej})-a_{ej}\|<\delta-\delta',\]
\[C\big\|\theta_{F',m'}\big|_{Y_{F,m}}-\id\big|_{Y_{F,m}}\|<\delta.\]
for $1\leq j,k\leq m$ and $s_{1},\cdots,s_{k}\in F.$ 

	Suppose that $\delta',F',m'$ are so chosen.  If $T\in \Hom_{\Gamma}(S,F',m',\delta',\sigma_{i})$ and $\phi=T\circ q\big|_{Y_{F',m'}}\circ \theta_{F',m'}$ then,
\[\rho_{V_{i}}(\alpha_{S}(T)-\alpha_{\F}(\phi))\leq C(C^{2}+1)\varepsilon+\rho_{V_{i}}(\chi_{\{j:j\leq r\}}(\alpha_{S}(T)-\alpha_{\F}(\phi)))\]
and if $j\leq r,$
\[\|\alpha_{S}(T)(j)-\alpha_{\F}(\phi)(j)\|=\|T(x_{j})-T\circ q\circ \theta_{F,l}(a_{ej})\|\leq C\kappa +\|T(x_{j})-T\circ q(a_{ej})\|=C\kappa.\]
Thus
\[\rho_{V_{i}}(\alpha_{S}(T)-\alpha_{\F}(\phi))\leq (C^{2}+C+1)\varepsilon.\]
Therefore
\[\widehat{d}_{(C^{2}+C+2)\varepsilon}(\Hom_{\Gamma}(S,F',m',\delta',\sigma_{i}),\rho)\leq \widehat{d}_{\varepsilon}(\Hom_{\Gamma}(\F,F,m,\delta,\sigma_{i})_{C^{2}},\rho).\]
Since $F',m'$ can be made arbitrary large and $\delta'$ arbitrarily small, this implies
\[f.\dim_{\Sigma}(S,\rho,(C^{2}+2C+1)\varepsilon)\leq \limsup_{i}\frac{1}{\dim V_{i}}\widehat{d}_{\varepsilon}(\Hom_{\Gamma}(\F,F,m,\delta,\sigma_{i})_{C^{2}},\rho),\]
taking the limit supremum over $(F,m,\delta)$ and then the supremum over $\varepsilon>0,$ 
\[f.\dim_{\Sigma}(S,\rho)\leq f.\dim_{\Sigma}(q,\Gamma).\]

	For the opposite inequality, fix $1>\varepsilon>0$ and let $r,\kappa,E,l$ be as before.  
 Fix $E\subseteq F\subseteq \Gamma$ finite, $m\geq \max(r,l)$ and $\delta<\min(\kappa,\varepsilon).$ 

	By Lemma $\ref{L:section}$ we can find $\delta'<\delta,$ and $F\subseteq F'\subseteq\Gamma$ finite and $m\leq m'\in \NN$ such that if $W$ is a Banach space and 
\[T\colon Y_{F',m'}\to W\]
has
\[\|T\|\leq 1,\]
\[\|T\big|_{\ker(q)\cap Y_{F',m'}}\|\leq \delta',\]
then there is a $\phi\colon X_{F,m}\to W$ such that
\[\|T(a_{s_{1}\cdots s_{k}j})-\phi(s_{1}\cdots s_{k}x_{j})\|\leq \delta,\mbox{ for $1\leq j,k\leq m,s_{1},\cdots,s_{k}\in F$}\]
and $\|\phi\|\leq 2C.$ 

Fix $T\in \Hom_{\Gamma}(\F,F',m',\delta',\sigma_{i}),$ and choose $\phi\colon X_{F,m}\to V_{i}$ such that $\|\phi\|\leq 2C$ and
\[\|T(a_{s_{1}\cdots s_{k}j})-\phi\circ q(a_{s_{1}\cdots s_{k}j})\|\leq \delta,\mbox{ for $1\leq j,k\leq m,s_{1},\cdots,s_{k}\in F$}.\]
Thus for $1\leq j,k\leq m$ and $s_{1},\cdots,s_{k}\in F$ we have
\begin{align*}
\|\phi(s_{1}\cdots s_{k}x_{j})-\sigma_{i}(s_{1})\cdots \sigma_{i}(s_{k})\phi(x_{j})\|&\leq 2\delta\\
&+\|T(a_{s_{1}\cdots s_{k}j})-\sigma_{i}(s_{1})\cdots \sigma_{i}(s_{k})T(a_{ej})\| \\
&<2\delta+\delta'\\
&<3\delta.
\end{align*}

Thus $\phi\in \Hom_{\Gamma}(S,F,m,3\delta,\sigma_{i})_{2C}.$ Furthermore, for $1\leq j\leq r$ 
\[\|\alpha_{S}(T)(j)-\alpha_{\F}(\phi)(j)\|=\|T(a_{ej})-\phi\circ q(a_{ej})\|\leq \kappa,\]
so 
\[\rho_{V_{i}}(\alpha_{\F}(T)-\alpha_{S}(\phi))\leq \varepsilon+(2C^{2}+C)\varepsilon=(2C^{2}+C+1)\varepsilon.\]
Thus
\[f.\dim_{\Sigma}(\F,(2C^{2}+C+2)\varepsilon,\rho)\leq \limsup_{i}\frac{1}{\dim V_{i}}\widehat{d}_{\varepsilon}\left(\Hom_{\Gamma}(S,F,m,3\delta,\sigma_{i})_{2C},\rho \right),\]
and since $F,m,\delta,\varepsilon$ are arbitrary this completes the proof.

\end{proof}

Because of the preceding Lemma and Corollary $\ref{C:ind}$, we know that 
\[f.\dim_{\Sigma}(S,\rho),\dim_{\Sigma}(q,\Gamma)\]
 only depend upon the action of $\Gamma$ on $X,$ and are equal. Because of this we will use
\[\dim_{\Sigma}(X,\Gamma)=f.\dim_{\Sigma}(S,\rho)=\dim_{\Sigma}(q,\Gamma)\]
for any dynamically generating sequence $S,$ and any bounded linear surjective map $q\colon Y\to X,$ where $Y$ has the bounded approximation property. We similarly define $\underline{\dim}_{\Sigma}(X,\Gamma).$

We now prove a lemma which allows us to treat the limit supremum over $(F,m,\delta)$ in the definition of $f.\dim_{\Sigma}(S,\rho)$ as a limit.

\begin{lemma}\label{L:dimlimit} Let $(X,\Gamma,\Sigma=(\sigma_{i}\colon \Gamma\to \Isom(V_{i}))$ be a dimension triple, fix a dynamically generating sequence $S$ in $X$ and $\rho$ a  product norm. Then
\[f.\dim_{\Sigma}(S,\rho)=\sup_{\varepsilon>0}\liminf_{(F,m,\delta)}\limsup_{i}\frac{1}{\dim V_{i}}\widehat{d}_{\varepsilon}(\Hom_{\Gamma}(S,F,m,\delta,\sigma_{i}),\rho),\]
\[\underline{f.\dim}_{\Sigma}(S,\rho)=\sup_{\varepsilon>0}\limsup_{(F,m,\delta)}\liminf_{i}\frac{1}{\dim V_{i}}\widehat{d}_{\varepsilon}(\Hom_{\Gamma}(S,F,m,\delta,\sigma_{i}),\rho).\]
\end{lemma}

\begin{proof} Let $S=(x_{j})_{j=1}^{\infty}.$ We do the proof for $\dim$ only, the proof for $\underline{\dim}$ is the same. Fix $\varepsilon>0$ and choose $k\in \NN$ such that if $\|f\|_{\infty}\leq 1+\sup_{j\in \NN}\|x_{j}\|$ and $f$ is supported on $\{n:n\geq k\}
,$ then $\rho(f)<\varepsilon.$ It suffices to show that 
\[f.\dim_{\Sigma}(S,\rho)\leq \sup_{\varepsilon}\liminf_{(F,m,\delta)}\limsup_{i}\frac{1}{\dim V_{i}}\widehat{d}_{\varepsilon}(\Hom_{\Gamma}(S,F,m,\delta,\sigma_{i}),\rho).\]
	Fix $F\subseteq \Gamma$ finite $m\geq k,\delta>0.$ Then for any $F\subseteq F'\subseteq \Gamma$ finite, $m'\geq m,\delta'<\delta$ and $\psi\in \Hom_{\Gamma}(S,F',m',\delta',\sigma_{i})$ we have $\psi\in \Hom_{\Gamma}(S,F,m,\delta,\sigma_{i}).$ 

	Furthermore if $f,g\in l^{\infty}(\NN,V_{i})$ are defined by 
\[f(j)=\chi_{\{n\leq m\}}(j)\psi(x_{j}),g(j)=\chi_{\{n\leq m'\}}\psi(x_{j})\]
 then
\[\rho(j\mapsto \|f(j)-g(j)\|)<\varepsilon.\]
Thus
\[\widehat{d}_{2\varepsilon}(\Hom_{\Gamma}(S,F',m',\delta',\sigma_{i}),\rho)\leq \widehat{d}_{\varepsilon}(\Hom_{\Gamma}(S,F,m,\delta,\sigma_{i}),\rho).\]
Therefore
\[f.\dim_{\Sigma}(S,2\varepsilon,\rho)\leq \limsup_{i}\frac{1}{\dim V_{i}}\widehat{d}_{\varepsilon}(\Hom_{\Gamma}(S,F,m,\delta,\sigma_{i}),\rho).\]
Since $F,m,\delta$ were arbitrary 
\[f.\dim_{\Sigma}(S,2\varepsilon,\rho)\leq\liminf_{(F,m,\delta)}\limsup_{i}\frac{1}{\dim V_{i}}\widehat{d}_{\varepsilon}(\Hom_{\Gamma}(S,F,m,\delta,\sigma_{i}),\rho),\]
and taking the supremum over $\varepsilon>0$ completes the proof.

\end{proof}

\section{Main Properties of $\dim_{\Sigma}(X,\Gamma)$}
The first property that we prove is that dimension is decreasing under surjective maps, as in the usual case of finite-dimensional vector spaces.

\begin{proposition}\label{P:surjection} Let $(Y,\Gamma,\Sigma=(\sigma_{i}\colon \Gamma\to \Isom(V_{i}))),(X,\Gamma,\Sigma)$ be two dimension triples. Suppose that there is a $\Gamma$-equivariant bounded linear map $T\colon Y\to X,$ with dense image. Then
\[\dim_{\Sigma}(X,\Gamma)\leq \dim_{\Sigma}(Y,\Gamma).\]
\[\underline{\dim}_{\Sigma}(X,\Gamma)\leq \underline{\dim}_{\Sigma}(Y,\Gamma).\]
\end{proposition}

\begin{proof} Let $S'=(y_{j})_{j=1}^{\infty}$ be a dynamically generating sequence for $Y.$ Let $S=(T(x_{j}))_{j=1}^{\infty},$ then $S$ is dynamically generating for $X.$ Then
\[\Hom_{\Gamma}(S,F,m,\delta,\sigma_{i})\circ T\subseteq \Hom_{\Gamma}(S',F,m,\delta,\sigma_{i})_{\|T\|},\]
and 
\[\alpha_{S'}(\phi\circ T)=\alpha_{S}(\phi),\]
so the proposition follows.

\end{proof}

We next show that dimension is subadditive under exact sequences. It turns out to be strong of a condition to require that dimension be additive under exact sequences. As noted in \cite{Gor} if $\dim_{\Sigma,l^{p}}$ is additive under exact sequences and
\[\dim_{\Sigma,l^{p}}(l^{p}(\Gamma)^{\oplus n},\Gamma)=n,\]
then we can write the Euler characteristic of a group as an alternating sum of dimensions of $l^{p}$ cohomology spaces. But torsion-free cocompact lattices in $SO(4,1)$ have positive Euler characteristic and their $l^{p}$ cohomology vanishes when $p$ is sufficiently large, so this would give a contradiction.

\begin{proposition}\label{P:dsumsubadd} Let $(V,\Gamma,\Sigma=(\sigma_{i}\colon \Gamma\to \Isom(V_{i})))$ be a dimension triple. Let $W\subseteq V$ be a closed $\Gamma$-invariant subspace. Then
\[\dim_{\Sigma}(V,\Gamma)\leq \dim_{\Sigma}(V/W,\Gamma)+\dim_{\Sigma}(W,\Gamma),\]
\[\underline{\dim}_{\Sigma}(V,\Gamma)\leq \underline{\dim}_{\Sigma}(V/W,\Gamma)+\dim_{\Sigma}(W,\Gamma),\]
\[\underline{\dim}_{\Sigma}(V^{\oplus n},\Gamma)\leq n\underline{\dim}_{\Sigma}(V,\Gamma).\]
\end{proposition}
\begin{proof}

	Let $S_{2}=(w_{j})_{j=1}^{\infty}$ be a dynamically generating sequence for $W,$ and let $S_{1}=(a_{j})_{j=1}^{\infty}$ be a dynamically generating sequence for $V/W.$ Let $x_{j}\in V,$ be such that $x_{j}+W=a_{j},$ and $\|x_{j}\|\leq 2\|a_{j}\|.$ Let $S$ be the sequence
\[x_{1},w_{1},x_{2},w_{2},\cdots.\]
	We shall use the product norm on $l^{\infty}(\NN)$ given by
\[\rho_{1}(f)=\sum_{j=1}^{\infty}\frac{1}{2^{j}}|f(j)|,\]
\[\rho_{2}(f)=\sum_{j=1}^{\infty}\frac{1}{2^{j}}|f(2j)|+\sum_{j=1}^{\infty}\frac{1}{2^{j}}|f(2j-1)|.\]

	Let $\varepsilon>0,$ and choose $m$ such that $2^{-m}<\varepsilon.$ Let $e\in F_{1}\subseteq \Gamma$ be finite,$m\leq m_{1}\in \NN,$ and $\delta_{1}>0.$ Let $\eta>0$ to be determined later. By Lemma $\ref{L:section},$ we can find a $\delta_{1}>\delta>0,$ a $F_{1}\subseteq E\subseteq \Gamma$ finite, and a $m\leq k\in \NN,$ so that if $X$ is a Banach space, and 
\[T\colon V_{E,2k}\to X\]
has $\|T\|\leq 2,$ and 
\[\|T\big|_{W\cap V_{E,2k}}\|\leq \delta,\]
then there is a $\phi\colon (V/W)_{F_{1},m_{1}}\to X$ with $\|\phi\|\leq 3,$ and
\[\|\phi(s_{1}\cdots s_{k}a_{j})-T(s_{1}\cdots s_{k}x_{j})\|<\delta_{1},\]
for all $1\leq j,k\leq m_{1},$ and $s_{1},\cdots,s_{k}\in F_{1}.$

	By finite-dimensionality, we can find a finite set $F'\supseteq E,m'\geq 2k,$ and a $0<\delta'<\delta_{1},$ so that if $T\colon V_{F',m'}\to X,$ satisfies
\[\|T(s_{1}\cdots s_{k}x_{j})\|<\delta'\]
for all $1\leq j,k\leq m',$ and $s_{1},\cdots,s_{k}\in F',$ then
\[\|T\big|_{W\cap V_{E,2k}}\|\leq \delta.\]

	Define
\[R\colon \Hom_{\Gamma}(S,F',2m',\delta',\sigma_{i})\to \Hom_{\Gamma}(S_{2},F',m',\delta',\sigma_{i})\]
by
\[R(T)=T\big|_{W_{F',m'}}.\]

	Find
\[\Theta\colon \im(R)\to \Hom_{\Gamma}(S,F',2m',\delta',\sigma_{i})\]
so that  $R\circ \Theta=\id.$

	Then
\[(T-\theta(R(T))(s_{1}\cdots s_{k}w_{j})=0,\]
for all $1\leq j,k\leq m',$ and $s_{1},\cdots,s_{k}\in F'.$ Thus by assumption, we can find a 
\[\phi\colon (V/W)_{F_{1},m_{1}}\to V_{i},\]
so that $\|\phi\|\leq 3,$ and
\[\|\phi(s_{1}\cdots s_{k}a_{j})-(T-\theta(R(T)))(s_{1}\cdots s_{k}x_{j})\|<\delta_{1},\]
for all $1\leq j,k\leq m_{1},$  $s_{1},\cdots,s_{k}\in F_{1},$
in particular,
\[\|\phi(a_{j})-(T-\theta(R(T)))(x_{j})\|<\delta_{1},\]
for $1\leq j\leq m.$

	Thus whenever $1\leq j,k\leq m_{1},$  $s_{1},\cdots,s_{k}\in F_{1},$
\[\|\phi(s_{1}\cdots s_{k}a_{j})-\sigma_{i}(s_{1})\cdots \sigma_{i}(s_{k})\phi(a_{j})\|\leq 2\delta_{1}+2\delta'<4\delta_{1}.\]

	Now suppose that 
\[\alpha_{S_{2}}(\Hom_{\Gamma}(S_{2},F_{1},m_{1},\delta_{1},\sigma_{i}))\subseteq_{\varepsilon,\rho_{1,V_{i}}}G,\]
\[\alpha_{S_{1}}(\Hom_{\Gamma}(S_{1},F,m,4\delta_{1},\sigma_{i})_{3})\subseteq_{\varepsilon,\rho_{1,V_{i}}}F.\]
Let $E\subseteq l^{\infty}(\NN,V_{i})$ be the subspace consisting of all $h$ so that there are $f\in F,g\in G$ so that 
\[h(2k)=g(k),h(2k-1)=f(k).\]
	Then $\dim(E)=\dim(F)+\dim(G).$  It easy to see that 
\[\alpha_{S}(\Hom_{\Gamma}(S,F',m',\delta',\sigma_{i}))\subseteq_{3\varepsilon+\delta_{1},\rho_{2,V_{i}}}E.\]

	So if $\delta_{1}<\varepsilon,$ we find that 
\[\alpha_{S}(\Hom_{\Gamma}(S,F_{1},m_{1},\delta',\sigma_{i}))\subseteq_{3\varepsilon}E.\]
	From this the first two inequalities follow.

	The last inequality is easier and its proof will only be sketched. Let $S=(x_{j})_{j=1}^{\infty}$ be a dynamically generating sequence for $X,$ and $y_{j}=x_{q}\otimes e_{r}$ if $j=nq+r,$ with $1\leq r\leq n,$ and $x_{q}\otimes e_{r}$ is the element of $X^{\oplus n}$ which is zero in all coordinates except for the $r^{th},$ where it is $x_{q}.$ If $F\subseteq \Gamma$ is finite $m\in \NN,\delta>0,$ then
\[\Hom_{\Gamma}(S,F,nm,\delta,\sigma_{i})\subseteq \Hom_{\Gamma}(S,F,m,\delta,\sigma_{i})^{\oplus n}.\]
The rest of the proof proceeds as above.
\end{proof}

	We note here that subadditivity is not true for \emph{weakly} exact sequences, that is sequences
\[\begin{CD} 0 @>>> X@>>> Y @>>> Z @>>> 0\end{CD},\]
where $X\to Y$ is injective, $\overline{\im(X)}=\ker(Y\to Z),$ and the image of $Y$ is dense in $Z.$ In fact, using $\FF_{n}$ for the free group on $n$ letters $a_{1},\cdots, a_{n},$ it is known that the map
\[\parital\colon l^{1}(\FF_{n})^{\oplus n}\to l^{1}(\FF_{n}),\]
given by
\[\partial (f_{1},\cdots,f_{n})(x)=\sum_{j=1}^{n}f_{j}(x)-\sum_{j=1}^{n}f_{j}(xa_{j}^{-1})\]
has dense image and is injective. We will show in our follow-up paper that
\[\underline{\dim}_{\Sigma,l^{1}}(l^{1}(\FF_{n})^{\oplus n},\FF_{n})=\dim_{\Sigma,l^{1}}(l^{1}(\FF_{n})^{\oplus n},\FF_{n})=n,\]
\[\underline{\dim}_{\Sigma,l^{1}}(l^{1}(\FF_{n}),\FF_{n})=\dim_{\Sigma,l^{1}}(l^{1}(\FF_{n}),\FF_{n})=1,\]
this gives a counterexample to subadditivity under weakly exact sequences. This also gives a counterexample to monotonicity under injective maps, though once should note in this case that the map defined above does not have closed image.

	For $2\leq p\leq \infty,$ we  have a lower bound for direct sums, whose proof requires a few more lemmas.

\begin{lemma}[ \cite{Voi}, Lemma 8.5]\label{L:Voic} Let $H_{1},H_{2}$ be Hilbert spaces and let $H=H_{1}\oplus H_{2}$ and let $\Omega_{j}\subseteq H_{j}$  and suppose $C_{1},C_{2}>0$ are such that $C_{1}\leq \|\xi\|\leq C_{2},$ for all $\xi\in \Omega_{j}.$ If $0<\delta<C_{1},$ then 
\[d_{C_{2}^{-1}\delta}(\Omega_{1}\oplus 0\cup 0\oplus \Omega_{2})\geq d_{C_{1}^{-1}\sqrt{5\delta}}(\Omega_{1})+d_{C_{1}^{-1}\sqrt{5\delta}}(\Omega_{2}).\]
\end{lemma}\begin{proof} By replacing $\Omega_{j}$ with 
\[\left\{\frac{\xi}{\|\xi\|}:\xi\in \Omega_{j}\right\}\]
 we may assume $C_{1}=C_{2}=1.$  Let $P_{i}$ be the projection onto each $H_{i},$ and set $\Omega=(\Omega_{1}\oplus 0)\cup (0\oplus \Omega_{2}).$ Suppose that $V$ is a subspace such that $\Omega\subseteq_{\delta} V,$ and let $Q$ be the projection onto $V$ and $T=QP_{1}Q\big|_{V}.$  Define
\[\Omega_{1}'=Q(\Omega_{1}\oplus 0),\Omega_{2}'=Q(0\oplus \Omega_{2}).\]
 For $\xi\in \Omega$ we have 
\[\|(1-Q)\xi\|\leq \delta\]
thus for $\xi\in \Omega_{1}\oplus \{0\}$
\[\ip{TQ\xi,Q\xi}=\ip{QP_{1}Q\xi,Q\xi}=\|P_{1}Q\xi\|^{2}\geq (\|\xi\|-\|P_{1}(1-Q)\xi\|)^{2}\geq (1-\delta)^{2}.\]

	So if $T=\int_{[0,1]}t\,dE(t)$ we have  with $\eta=Q\xi$
\[(\sqrt{1-\delta^{2}}-\delta)^{2}\leq \left\ip{\left(1-\frac{1}{2}E\left([0,1/2]\right)\right)\eta,\eta\right}\leq 1-\frac{1}{2}\left\|E([0,1/2])\eta\right\|^{2}.\]

	Thus
\[\|E([0,1/2])\eta\|^{2}\leq 2(1-(1-\delta)^{2})\leq 4\delta\]
i.e.
\[\|\eta-E((1/2,1])\eta\|^{2}\leq 4\delta.\]
Thus
\[\Omega_{1}'\subseteq_{2\sqrt{\delta}}E((1/2,1])V.\]
Similarly, because $QP_{2}Q\big|_{V}=1-T$ we have
\[\Omega_{2}'\subseteq_{2\sqrt{\delta}}E([0,1/2])V.\]

	For any projection $P'$ and any $x\in H$ we have $\|x-P'x\|^{2}=\|x\|^{2}-\|P'x\|^{2}.$ So for all $\xi\in \Omega_{1}\oplus 0$ we have since,  $QE((1/2,1])=E((1/2,1]) $ (and $E((1/2,1]Q=E((1/2,1])$ by taking adjoints), that
\[\|\xi-E((1/2,1])Q\xi\|^{2}=\|\xi-E((1/2,1])\xi\|^{2}=\|\xi\|^{2}-\|E((1/2,1])\xi\|^{2}=\]
\[\|\xi\|^{2}-\|Q\xi\|^{2}+\|Q\xi\|^{2}-\|E((1/2,1])\xi\|^{2}=\]
\[\|\xi-Q\xi\|^{2}+\|Q\xi-E((1/2,1])Q\xi\|^{2}\leq \delta^{2}+4\delta<5\delta.\]

	Thus with a similar proof for $\Omega_{2}$ we have
\[\Omega_{1}\oplus 0\subseteq_{\sqrt{5\delta}}E((1/2,1])V\]
\[0\oplus \Omega_{2}\subseteq_{\sqrt{5\delta}}E([0,1/2))V\]
since
\[V=E([0,1/2])V\oplus E\left((1/2,1]\right)V\]
the desired claim follows.

\end{proof}

\begin{lemma}\label{L:finite sequence} Let $(X,\Gamma,\Sigma)$ be a dimension triple.  Let $S$ be a dynamically generating sequence in $X,$ and $\rho$ a  product norm such that $\rho(f)\leq \rho(g)$ if $|f|\leq |g|$. Set
\[\rho^{(N)}(f)=\rho(\chi_{j\leq N}f).\]

 Then
\[f.\dim_{\Sigma}(S,\rho)=\lim_{N\to \infty}f.\dim_{\Sigma}(S,\rho^{(N)}),\]
\[\underline{f.\dim}_{\Sigma}(S,\rho)=\lim_{N\to \infty}\underline{f.\dim}_{\Sigma}(S,\rho^{(N)}).\]
\end{lemma}
\begin{proof} Let $\Sigma=(\sigma_{i}\colon \Gamma\to \Isom(V_{i})).$ Let $S=(x_{j})_{j=1}^{\infty},$ $C=\sup_{j}\|x_{j}\|.$

	Since $\rho^{(N)}\leq \rho,$ for any $\varepsilon>0$
\[f.\dim_{\Sigma}(S,\varepsilon,\rho^{(N)})\leq f.\dim_{\Sigma}(S,\varepsilon,\rho)\leq f.\dim_{\Sigma}(S,\rho),\]
thus
\[\limsup_{n\to \infty}f.\dim_{\Sigma}(S,\rho^{(n)})\leq f.\dim_{\Sigma}(S,\rho).\]

	For the opposite inequality, fix $\varepsilon>0.$ and choose $N$ such that $\rho(f)<\varepsilon$ if $f\in l^{\infty}(\NN,V_{i})$ is supported on $\{k:k\geq N\}$ and $\|f\|_{\infty}\leq C.$ Thus for $T\in B(X,V_{i}),$ and $f\in l^{\infty}(\NN,V_{i})$ with $\|T\|\leq 1,$ and $n\geq N$ we have
\[|\rho_{V_{i}}(\alpha_{S}(T)-\chi_{\{j\leq N\}})-(\rho^{(n)}_{V_{i}}(\alpha_{S}(T)-\chi_{\{j\leq N\}}f))|\leq |\rho_{V_{i}}(\chi_{\{k>n\}}\alpha_{S}(T))|\leq \varepsilon.\]
Thus for $n\geq N,$ 
\[f.\dim_{\Sigma}(S,2\varepsilon,\rho)\leq f.\dim_{\Sigma}(S,\varepsilon,\rho^{(n)})\leq  f.\dim_{\Sigma}(S,\rho^{(n)}),\]
so
\[f.\dim_{\Sigma}(S,2\varepsilon,\rho)\leq \liminf_{n\to \infty} f.\dim_{\Sigma}(S,\rho^{(n)}).\]

\end{proof}
For the next lemma, we recall the notion of the volume ratio of a finite-dimensional Banach space. Let $X$ be an $n$-dimensional real Banach space, which we will identify with $\RR^{n}$ with a certain norm. By an \emph{ellipsoid} in $\RR^{n}$ we mean a set which is the unit ball for some Hilbert space norm on $\RR^{n}.$ Let $B\subseteq \RR^{n}$ be the unit ball of $X.$ We define the volume ratio of $B,$ denoted $\vr(B)$ by
\[\vr(B)=\inf\left(\frac{\vol(B)}{\vol(D)}\right)^{1/n},\]
where the infimum runs over all ellipsoids $D\subseteq B.$ It is know that for any unit ball  $B$ of a Banach space norm on $\RR^{n},$ there is  an ellipsoid $D^{\max{}}$ such that $D^{\max{}}\subseteq B,$ and $D^{\max{}}$ has the largest volume of all such ellipsoids. So we have
\[\vr(B)=\left(\frac{\vol(B)}{\vol(D^{\max{}})}\right)^{1/n}.\]
The main property we will need to know about volume ratio is the following theorem.

\begin{theorem}[Theorem 6.1,\cite{Pis}] Let $B\subseteq \RR^{n}$ be the unit ball for a norm $\|\cdot\|$ on $\RR^{n}.$ Let  $D\subseteq B$ be an ellipsoid. Set
\[A=\left(\frac{\vol(B)}{\vol(D)}\right)^{1/n}.\]
Let $|\cdot|$ be a norm such that $D$ is the unit ball of $(\RR^{n},|\cdot|),$ in particular $\|\cdot\|\leq |\cdot|.$ Then for all $k=1,\cdots, n-1$ there is a subspace $F\subseteq \RR^{n}$ such that $\dim F=k$ and for every $x\in F$
\begin{equation}\label{E:equivalentnorm}
|x|\leq (4\pi A)^{\frac{n}{n-k}}\|x\|.
\end{equation}
Further if we let $G_{nk}$ be the Grassmanian manifold of $k$-dimensional subspaces of $\RR^{n},$ then
\[\PP(\{F\in G_{nk}:\mbox{ for all $x\in F$, equation $(\ref{E:equivalentnorm})$ holds}\})>1-2^{-n},\]
for the unique $O(n)$-invariant probability measure on $G_{nk}.$ 
\end{theorem}

What we will actually use is the following corollary.

\begin{cor}\label{C:equivquo} Let $B\subseteq \RR^{n}$ be the unit ball for a norm $\|\cdot\|$ on $\RR^{n},$ and let $B^{o}$ be its polar. Let $D\subseteq B^{o}$ be an ellipsiod. Set
\[A=\left(\frac{\vol(B^{o})}{\vol(D^{o})}\right)^{1/n}.\]
Let $|\cdot|$ be a norm such that $D$ is the unit ball of $(\RR^{n},|\cdot|),$ in particular $|\cdot|\leq \|\cdot\| .$ Then for all $k=1,\cdots, n-1$ there is a subspace $F\subseteq \RR^{n}$ such that $\dim F=k$ and for every $x\in \RR^{n}/F^{\perp}$
\begin{equation}\label{E:equivalentnorm2}
\|x\|_{(\RR^{n}/F^{\perp},\|\cdot\|)}\leq (4\pi A)^{\frac{n}{n-k}}|x|_{(\RR^{n}/F^{\perp},|\cdot|)},
\end{equation}
where we use $\|\cdot\|_{(\RR^{n}/F^{\perp},\|\cdot\|)}$ for the quotient norm induced by $\|\cdot\|$ and similarly for $|\cdot|.$ Further,
\[\PP(\{F\in G_{nk}\colon \mbox{ for all $x\in F,$ equation $(\ref{E:equivalentnorm2})$ holds}\})>1-2^{-n}.\]
\end{cor}
\begin{proof} This is precisely the dual of the above theorem.
\end{proof}

Here is the main application of the above corollary to dimension theory.

\begin{theorem}\label{T:sumsubadd} Let $\Gamma$ be a countable group with a uniformly bounded action on separable Banach spaces $X,Y.$ Let $\Sigma=(\sigma_{i}\colon \Gamma\to \Isom(V_{i}))$ with $\dim V_{i}<\infty.$  Suppose that $V_{i}$ is the complexification of a real Banach space $V_{i}'$ such that 
\[\sup_{i}\vr((V_{i}')^{*})<\infty,\]
and there are constants $C_{1},C_{2}>0$ so that
\[C_{1}(\|x\|_{V_{i}'}+\|y\|_{V_{i}'})\leq \|x+iy\|\leq C_{2}(\|x\|_{V_{i}'}+\|y\|_{V_{i}'}),\]
for all $x,y\in V_{i}.$ 
Then the following inequalities hold,
\[\underline{\dim}_{\Sigma}(X\oplus Y_,\Gamma)\geq \underline{\dim}_{\Sigma}(X,\Gamma)+\underline{\dim}_{\Sigma,X_{i}}(Y,\Gamma),\]
\[\dim_{\Sigma}(Y_{1}\oplus Y_{2},\Gamma)\geq \dim_{\Sigma}(X,\Gamma)+\underline{\dim}_{\Sigma}(Y,\Gamma),\]
\[\dim_{\Sigma}(Y^{\oplus n},\Gamma)\geq n \dim_{\Sigma}(Y,\Gamma),\]

\end{theorem}
\begin{proof}  We will do the proof for $\dim$ only, the proof of the other claims are the same. Let $S=(x_{n})_{n=1}^{\infty},T=(y_{n})_{n=1}^{\infty}$ be   dynamically generating sequences, enumerate $S\oplus \{0\}\cup \{0\}\oplus T$ by $x_{1},y_{1},x_{2},y_{2},\cdots,$ and fix integers $k,m.$   By Lemma $\ref{L:finite sequence},$ it suffices to show that for fixed $m,k\in \NN,$ and for the pseudonorms $\rho,\rho_{1},\rho_{2}$ on $l^{\infty}(\NN)$ given by
\[\rho(f)=\left(\sum_{j=1}^{m+k}|f(j)|^{2}\right)^{1/2},\]
\[\rho_{1}(f)=\left(\sum_{j=1}^{m}|f(j)|^{2}\right)^{1/2},\]
\[\rho_{2}(f)=\left(\sum_{j=1}^{k}|f(j)|^{2}\right)^{1/2},\]
 we have
\[f.\dim_{\Sigma}(S\oplus 0\cup 0\oplus T,\rho)\geq f.\underline{\dim}_{\Sigma}(S,\rho_{1})+f.\dim_{\Sigma}(T,\rho_{2}).\]

 Fix $\kappa,\varepsilon>0$ and fix $\eta>0$ which will depend upon $\kappa,\varepsilon$ in a manner to be determined later. By Corollary $\ref{C:equivquo}$ there is a constant $A,$ which depends only on $\kappa,C_{1},C_{2}$  Hilbert space norms $|\cdot|_{i}$ on $X_{i},$ and finite dimensional complex subspaces $F_{i}\subseteq V_{i}^{*}$ of complex dimension $\lfloor{(1-\kappa)(\dim V_{i})\rfloor}$ such that
\[\frac{1}{A}|x|_{i}\leq \|x\|\leq \|x\|\leq A|x|_{i}\]
 for all $x\in V_{i}/F_{i}^{\perp}.$ Here, as in the Corollary $\ref{C:equivquo},$ we abuse notation by using $\|x\|$ for the norm on $X_{i}/F_{i}^{\perp}$ induced by $\|\cdot\|,$ and similarly for $|\cdot|.$ 

	For $m'\geq m\in \NN,\delta>0$ and $F\subseteq \Gamma$ finite we have
\[\Hom_{\Gamma}(S,F,2m',\delta,\sigma_{i})\oplus \Hom_{\Gamma}(T,F,2m',\delta,\sigma_{i})_{2}\subseteq \Hom_{\Gamma}((S\oplus \{0\})\cup(\{0\}\oplus T),F,m',2\delta,\sigma_{i}).\]
Thus
\[\widehat{d}_{\eta}\left(\Hom_{\Gamma}((S\oplus \{0\})\cup(\{0\}\oplus T),F,2m',2\delta,\sigma_{i})_{2},\rho\right)\geq\]
\[\widehat{d}_{\eta}\left(\Hom_{\Gamma}(S,F,2m',\delta,\sigma_{i})\oplus \Hom_{\Gamma}(T,F,2m',\delta,\sigma_{i})_{2},\rho\right).\]

	Let 
\[K_{1}=\{(T(x_{1}),\cdots,T(x_{m})):T\in \Hom_{\Gamma}(S,F,2m',\delta,\sigma_{i})\}\]
\[K_{2}=\{(S(y_{1}),\cdots,S(y_{k})):S\in \Hom_{\Gamma}(S,F,2m',\delta,\sigma_{i})\}.\]
Then, by definition,
\[\widehat{d}_{\eta}\left(\Hom_{\Gamma}(S,F,2m',\delta,\sigma_{i})\oplus \Hom_{\Gamma}(T,F,2m',\delta,\sigma_{i}),\rho\right)=\]
\[d_{\eta}\left(K_{1}\oplus K_{2},\|\cdot\|^{\oplus m}\oplus  \|\cdot\|^{\oplus k}\right)\]
where we use the $l^{2}$-direct sum.

	Let $\pi_{i}\colon V_{i}\to V_{i}/F_{i}^{\perp}$ be the quotient map and let
\[G_{j}=\pi_{i}^{\oplus l}(K_{j}),\]
where $l=m$ if $j=1,$ and $l=k$ if $j=2.$ 

	Then
\[d_{\eta}\left(K_{1}\oplus K_{2},\|\cdot\|^{\oplus m}\oplus  \|\cdot\|^{\oplus k}\right)\geq d_{\eta}\left(G_{1}\oplus G_{2},\|\cdot\|^{\oplus m}\oplus  \|\cdot\|^{\oplus k}\right)\geq\]
\[ d_{A \eta}\left(G_{1}\oplus G_{2},|\cdot|_{i}^{\oplus m}\oplus |\cdot|_{i}^{\oplus k}\right).\]
Set
\[B_{i}=\left\{x\in G_{i}:lA\geq |x|\geq A\frac{\varepsilon}{4}\right\},\]
where $l=m$ if $i=1,$ and $l=k$ if $i=2.$ 

	Then
\begin{align*}
d_{A\eta}\left(G_{1}\oplus G_{2},|\cdot|_{i}^{\oplus m}\oplus |\cdot|_{i}^{\oplus k}\right)&\geq d_{\max(l,m)(\varepsilon/4)^{-1}\sqrt{5\eta A\max(l,m)}}(B_{1},|\cdot|^{\oplus m})\\
&+d_{\max(l,m)(\varepsilon/4)^{-1}\sqrt{5A\eta \max(l,m)})}(B_{2},|\cdot|^{\oplus k}).
\end{align*}
Setting $\eta=\frac{\varepsilon^{4/3}}{A\max(l,m)\cdot 5^{1/3}}$ we have
\begin{align*}
d_{\eta}(K_{1}\oplus K_{2},\|\cdot\|^{\oplus m}\oplus \|\cdot\|^{\oplus k})&\geq d_{\frac{\varepsilon}{A}}(B_{1},|\cdot|^{\oplus m})+d_{\frac{\varepsilon}{A}}(B_{2},|\cdot|^{\oplus k})\\
&\geq d_{\varepsilon}(B_{1},\|\cdot\|^{\oplus k})+d_{\varepsilon}(B_{2},\|\cdot\|^{\oplus k}).
\end{align*}
Since $B_{i}\supseteq\{x\in C_{i}:\|x\|\geq \frac{\varepsilon}{4}\}$ we have
\[d_{\varepsilon}(B_{1},\|\cdot\|^{\oplus k})+d_{\varepsilon}(B_{2},\|\cdot\|^{\oplus k})=d_{\varepsilon}(G_{1},\|\cdot\|^{\oplus k})+d_{\varepsilon}(G_{2},\|\cdot\|^{\oplus k}).\]

	Let $E_{i}\subseteq (V_{i}/F_{i}^{\perp})^{\oplus l}$ be a linear subspace of minimal dimension which $\varespilon$-contains $C_{i}$ with respect to $\|\cdot\|^{\oplus l}$ ($l=k,$ if $i=1,$ and $l=m$ if $i=2.$) Let $\widetilde{E_{i}}\subseteq V_{i}$ be a linear subspace such that $\dim E_{i}=\dim \widetilde{E_{i}}$ and $\pi_{i}^{\oplus l}(\widetilde{E_{i}})=E_{i}.$ Set $W_{i}=\widetilde{E_{i}}+F_{i}^{\oplus l}.$ Then  $W_{i}$ has dimension at most $\dim E_{i}+lc_{i}$ with $\lim_{i\to \infty}\frac{c_{i}}{\dim V_{i}}=\kappa,$ since $\dim V_{i}\to \infty,$ and $K_{i}\subseteq_{\varepsilon,\|\cdot\|} V_{i}.$  Thus 
\[d_{\varepsilon}(G_{i},\|\cdot\|^{\oplus l})\geq \widehat{d}_{\varepsilon}(K_{i},\|\cdot\|^{\oplus l})-lc_{i}.\]
Since $\varepsilon\to 0$ as $\eta\to 0$ (and vice versa) we conclude that
\[\dim_{\Sigma}(S_{1}\oplus S_{2},\Gamma,\|\cdot\|_{S,T,i})\geq -\kappa(k+m)+\dim_{\Sigma}(S_{1},\Gamma,\|\cdot\|_{S,i})+\underline{\dim}_{\Sigma}(Y_{2},\Gamma,\|\cdot\|_{T,i}).\]
Since $\kappa$ is arbitrary this proves the desired inequality.

\end{proof}
\begin{cor}\label{C:l^{p}supadd} Let $2\leq p<\infty.$

(a)  Let $\Gamma$ be a sofic group with uniformly bounded actions on separable Banach spaces $X,Y$ and let $\Sigma$ be a sofic approximation. Then
\[\dim_{\Sigma,l^{p}}(X\oplus Y,\Gamma)\geq \dim_{\Sigma,l^{p}}(X,\Gamma)+\underline{\dim}_{\Sigma,l^{p}}(Y,\Gamma)\]
\[\underline{\dim}_{\Sigma,l^{p}}(X\oplus Y,\Gamma)\geq \underline{\dim}_{\Sigma,l^{p}}(X,\Gamma)+\underline{\dim}_{\Sigma,l^{p}}(Y,\Gamma)\]

(b) Let $\Gamma$ be an $\R^{\omega}$-embeddable group with uniformly bounded actions on separable Banach spaces $X,Y$ and let $\Sigma$ be an embedding sequence. Then
\[\dim_{\Sigma,S^{p}}(X\oplus Y,\Gamma)\geq \dim_{\Sigma,S^{p}}(X,\Gamma)+\underline{\dim}_{\Sigma,S^{p}}(Y,\Gamma)\]
\[\underline{\dim}_{\Sigma,S^{p}}(X\oplus Y,\Gamma)\geq \underline{\dim}_{\Sigma,S^{p}}(X,\Gamma)+\underline{\dim}_{\Sigma ,S^{p}}(Y,\Gamma).\]
\end{cor}
\begin{proof} For $1\leq q\leq \infty,$ let $B_{q}$ be the unit ball of $L^{q}(\{1,\cdots,n\},\mu_{n})$ where $\mu_{n}$ is the uniform  measure.

 It is known that for all $q,$
\[\inf_{n}\left(\frac{\vol(B_{q})}{\vol(B_{2})}\right)^{1/n}>0,\]
\[\sup_{n}\left(\frac{\vol(B_{q})}{\vol(B_{2})}\right)^{1/n}<\infty,\]
(see the computation on page 11 of \cite{Pis}).
 Similarly if we let $C_{q}$ be the unit ball of $\{A\in M_{n}(\CC):A=A^{*}\}$ in the norm $\|\cdot\|_{L^{p}(\frac{1}{n}\Tr)},$ it is known that for all $q,$ 
\[\inf_{n}\left(\frac{\vol(C_{q})}{\vol(C_{2})}\right)^{1/n}>0,\]
\[\sup_{n}\left(\frac{\vol(C_{q})}{\vol(C_{2})}\right)^{1/n}<\infty,\]
(see \cite{TJ})
Apply the preceding theorem.
\end{proof}

 We note one last property of $l^{2}$-dimension for representations, which will be used in our follow-up paper to show that our dimension agrees with von Neumann dimension in the $l^{2}$-case.

\begin{proposition} Let  $H$  be a separable unitary representation of a $\R^{\omega}$-embeddable group $\Gamma.$ Let $\Sigma$ be an embedding sequence of $\Gamma.$ Suppose that $H=\overline{\bigcup_{k=1}H_{k}}$ with $H_{k}$ increasing, closed invariant subspaces, and that each $H_{k}$ has a finite dynamically generating sequence.  Then 
\[\dim_{\Sigma,l^{2}}(H,\Gamma)=\sup_{k}\dim_{\Sigma,l^{2}}(H_{k},\Gamma),\]
\[\underline{\dim}_{\Sigma,l^{2}}(H,\Gamma)=\sup_{k}\underline{\dim}_{\Sigma,l^{2}}(H_{k},\Gamma).\]

\end{proposition}
 
\begin{proof} We will do the proof for $\dim$ only, the other cases are the same. By Proposition $\ref{P:dsumsubadd}$ we know that $\dim_{\Sigma,l^{2}}$ is monotone for unitary representations, so we only need to show
\[\dim_{\Sigma,l^{2}}(H,\Gamma)\geq\sup_{k}\dim_{\Sigma,l^{2}}(H_{k},\Gamma).\]

	Let $\{\xi^{(k)}_{1},\cdots,\xi^{(k)}_{r_{k}}\}$ be unit vectors which dynamically generate $H_{k}.$ Let $S_{N}$ be the sequence
\[\xi^{(1)}_{1},\cdots,\xi^{(1)}_{r_{1}},\xi^{(2)}_{1},\cdots,\xi^{(2)}_{r_{2}},\cdots,\xi^{(N)}_{1},\cdots,\xi^{(N)}_{r_{N}},\]
i.e. the $l^{th}$ term of $S_{N}$ is
\[\xi^{(i)}_{q_{l}}\]
if $i$ is the largest integer such that
\[C_{i}=\sum_{j\leq i}r_{j}<l,\]
and
\[q_{l}=l-\sum_{j\leq i} r_{j}.\]
Let $S$ be the sequence obtained by the infinite concatenation of the $S_{N}$'s. We will use $S_{N}$ to compute $\dim_{\Sigma,l^{2}}(H_{N},\Gamma)$ and $S$ to compute $\dim_{\Sigma,l^{2}}(H,\Gamma),$ we also use the pseudornorms
\[\|T\|_{S,i}=\sum_{j=1}^{\infty}\frac{1}{2^{j}}\|T(\xi_{j})\|\]
\[\|T\|_{S_{N},i}=\sum_{j=1}^{\infty}\frac{1}{2^{j}}\|T(\xi_{j})\|.\]

Fix $\varepsilon>0,$ and let $M$ be such that $2^{-M}<\varepsilon.$  Suppose $F\subseteq \Gamma$ is finite,$\delta>0$ and $m\in \NN$ with $m> C_{M}.$ Let $P_{M}\in B(H)$ be the projection onto $H_{M}.$  Suppose $V$ is a subspace of $B(H_{M},\CC^{d_{i}})$ of minimal dimension such that
\[\Hom_{\Gamma}(S_{M},F,m,\delta,\sigma_{i})\subseteq_{\varepsilon,\|\cdot\|_{S,i}}V,\]
let $\widetidle{V}\subseteq B(H,\CC^{d_{i}})$ be the image of $V$ under the map $T\to T\circ P_{M}.$ If $T\in \Hom_{\Gamma,l^{2}(d_{i})}(S,F,m,\delta,\sigma_{i})$ then $\widetidle{T}=T\big|_{H_{M}}$ is in $\Hom_{\Gamma}(S_{M},F,m,\delta,\sigma_{i}),$ and there exists $\phi\in V$ such that $\|\phi-\widetidle{T}\|_{S_{M},i}<\varepsilon.$ Then
\[\|\phi\circ P-T\|_{S,i}\leq 2\sum_{n=C_{M}+1}^{\infty}\frac{1}{2^{n}}+\|\phi-\widetilde{T}\|_{S_{M},i}\leq 2^{-m+1}+\varepsilon\leq 3\varepsilon.\]
Thus
\[\Hom_{\Gamma}(S,F,m,\delta,\sigma_{i})\subseteq_{3\varepsilon,\|\cdot\|_{S,i}}\widetilde{V},\]
so
\[d_{3\varepsilon}(\Hom_{\Gamma}(S_{M},F,m,\delta,\sigma_{i}),\|\cdot\|_{S,i})\leq d_{\varepsilon}(\Hom_{\Gamma}(S_{M},F,m,\delta,\sigma_{i}),\|\cdot\|_{S_{M},i}).\]
Thus
\[\dim_{\Sigma,l^{2}}(S,\Gamma,3\varepsilon,\|\cdot\|_{S,i,2})\leq \dim_{\Sigma,l^{2}}(S_{M},3\varepsilon,\|\cdot\|_{S,i,2})\leq \sup_{M}\dim_{\Sigma,l^{2}}(\pi_{M})\]
and similarly for $\underline{\dim}.$
Taking the supremum over $\varepsilon>0$ completes the proof.
\end{proof}

\begin{cor}\label{C:dinfsumadd} Let $\Gamma$ be a  $\R^{\omega}$-embeddable group, and let  $\Sigma=\left(\sigma_{i}\colon \Gamma\to U(d_{i})\right)$ be an embedding sequence. Let $\pi_{k}\colon \Gamma\to U(H_{k})$ be a representations of $\Gamma$ such that each $\pi_{k}$ has a finite dynamically generating sequence. Then
\[\dim_{\Sigma,l^{2}}\left(\bigoplus_{k=1}^{\infty}\pi_{k}\right)\leq \sum_{k=1}^{\infty}\dim_{\Sigma,l^{2}}(\pi_{k})\]
\[\underline{\dim}_{\Sigma,l^{2}}\left(\bigoplus_{k=1}^{\infty}\pi_{k}\right)\geq \sum_{k=1}^{\infty}\underline{\dim}_{\Sigma,l^{2}}(\pi_{k}).\]
\end{cor}
\begin{proof} The corollary is a simple consequence of the above proposition and Theorem $\ref{T:sumsubadd}.$
\end{proof}

\section{ Computation of $\dim_{\Sigma,l^{p}}(l^{p}(\Gamma,V),\Gamma)$, and $\dim_{\Sigma,S^{p},conj}(l^{p}(\Gamma,V),\Gamma).$}\label{S:lp}

In this section we show that if $\Sigma$ is a sofic approximation of $\Gamma$ and $1\leq p\leq 2,$ then
\[\dim_{\Sigma,l^{p}}(l^{p}(\Gamma,V),\Gamma)=\dim V,\]
for $V$ finite dimensional. Similarly if $\Sigma$ is a  embedding sequence of $\Gamma$ and $1\leq p\leq 2,$ we show that
\[\dim_{\Sigma,S^{p},conj}(l^{p}(\Gamma,V),\Gamma)=\dim V,\]
\[\dim_{\Sigma,l^{2}}(l^{2}(\Gamma,l^{2}(n)),\Gamma)=n,\]
again for $V$ finite dimensional. 

	The proof for sofic groups will be relatively simple, but the proof for $\R^{\omega}$-embeddable groups requires a few more lemmas.

Let $\nu$ be the unique $U(n)$ invariant Borel probability measure on $S^{2n-1},$ for the next lemma we need that if $T\colon \CC^{n}\to \CC^{n}$ is linear, then
\[\frac{1}{n}\Tr(T)=\int_{S^{2n-1}}\ip{T\xi,\xi}\,d\nu(\xi).\]
This follows from the fact that $\Tr$ is, up to scaling, the unique linear functional on $M_{n}(\CC)$ invariant under conjugation by $U(n).$ 

Additonally, we will use the following concentration of measure fact (see \cite{Al} Page 295), if $f$ is a Lipschitz function on $S^{n-1},$ then
\[\mathbb P(|f-\mathbb E f|>t)\leq 4e^{\frac{-nt^{2}}{\|f\|_{\Lip}^{2}72\pi^{2}}}.\]

\begin{lemma}\label{L:choose} Let $\Gamma$ be a $\R^{\omega}$-embeddable group,  let $\sigma_{i}\colon \Gamma\to U(d_{i})$ be an embedding sequence, and fix $E\subseteq \Gamma$ finite,$ m\in \NN.$ For $j\in \{1,\cdots,m\},\xi,\eta\in S^{2d_{i}-1}$ define
\[T_{\xi,j}\colon l^{2}(\Gamma\times \{1,\cdots,m\})\to l^{2}(d_{i}),\]
\[T_{\xi,\eta,j}\colon l^{p}(\Gamma\times \{1,\cdots,m\})\to S^{p}(d_{i})  \]
by
\[T_{\xi,j}(f)=\sum_{s\in E}f(s,j)\sigma_{i}(s)\xi,\]
\[T_{\xi,\eta,j}(f)=\sum_{s\in E}f(s,j)\sigma_{i}(s)\xi\otimes \overline{\sigma_{i}(s)\eta}.\]
Then for any $\delta>0$ and  $1\leq p<\infty,$

(a)
\[\lim_{i\to \infty}\PP(\{\xi\in S^{2d_{i}-1}:\|T_{\xi,j}:l^{2}(\Gamma\times \{1,\cdots,m\})\to l^{2}(d_{i})\|<1+\delta\})=1,\]

(b)
\[\{(\xi,\eta)\in (S^{2d_{i}-1})^{2}:\|T_{\xi,\eta,j}\colon l^{p}(\Gamma\times \{1,\cdots,m\})\to S^{p}(d_{i})\|<1+\delta\}\supseteq A_{i}\times A_{i},\]
where $A_{i}\subseteq S^{2d_{i}-1}$ has $\nu(A_{i})\to 1.$ 
\end{lemma}
\begin{proof}  Let $\kappa>0$ which will depend upon $\delta>0,p$ in a manner to be determined later. Let
\[A=\bigcap_{s\ne t,s,t\in E}\{\xi\in S^{2d_{i}-1}:|\ip{\sigma_{i}(s)\xi,\sigma_{i}(t)\xi}|<\kappa\},\]
since
\[\int_{S^{2d_{i}-1}}\ip{\sigma_{i}(s)\xi,\sigma_{i}(t)\xi}\,d\nu(\xi)=\frac{1}{d_{i}}\Tr(\sigma_{i}(t)^{-1}\sigma_{i}(s))\to 0\]
for $s\ne t,$ the concentration of measure estimate mentioned before the Lemma implies that 
\[\nu(A)\to 1.\]
For the proof of $(a),(b)$ we prove that if $\xi,\eta\in A$ then 
\[\|T_{\xi,j}\|_{l^{2}\to l^{2}}\leq 1+\delta,\]
\[\|T_{\xi,\eta,j}\|_{l^{p}\to S^{p}}\leq 1+\delta,\]
if $\kappa>0$ is sufficiently small.

(a) For $f\in l^{2}(\Gamma\times \{1,\cdots,m\}),\xi \in A$ we have

\begin{align*}
\|T_{\xi,j}(f)\|_{2}^{2} &=\sum_{s,t\in E}f(s,j)\overline{f(t,j)}\ip{\sigma_{i}(s)\xi,\sigma_{i}(t)\xi}\\
&\leq \|f\chi_{E}\|_{2}^{2}+\sum_{s\ne t,s,t\in E}\|f\|_{2}^{2}\kappa\\
&\leq \|f\|_{2}^{2}(1+\kappa |E|^{2})\\
&\leq(1+\delta)\|f\|_{2}^{2}
\end{align*}

if $\kappa<\frac{\delta}{|E|^{2}}.$ 

(b) Fix $\varepsilon>0$ to be determined later. If $\kappa$ is sufficently small, then for any $(\xi,\eta) \in A^{2}$ we can find $(\xi_{s})_{s\in E}(\eta_{s})_{s\in E}$ such that $\ip{\xi_{s},\xi_{t}}=\delta_{s=t},\ip{\eta_{s},\eta_{t}}=\delta_{s=t}$ and
\[\|\xi_{s}-\sigma_{i}(s)\xi\|<\varepsilon,\|\eta_{s}-\sigma_{i}(s)\eta\|<\varepsilon.\]
	Then
\[\left\|T_{\xi,\eta,j}(f)-\sum_{s\in E}f(s)\xi_{s}\otimes \overline{\eta_{s}}\right\|_{p}\leq \|f\|_{p}\sum_{s\in E}(\|\xi_{s}-\sigma_{i}(s)\xi\|+\|\sigma_{i}(s)\eta-\eta_{s}\|)\leq 2|E|\varepsilon\|f\|_{p}.\]
Note that
\[\left|\sum_{s\in E}f(s)\xi_{s}\otimes \overline{\eta_{s}}\right|^{2}=\sum_{s,t\in E}\overline{f(s)}f(t)\ip{\xi_{t},\xi_{s}}\eta_{s}\otimes \overline{\eta_{t}}=\]
\[\sum_{s\in E}|f(s)|^{2}\eta_{s}\otimes \overline{\eta_{s}}.\]
Thus
\[\left\|\sum_{s\in E}f(s)\xi_{s}\otimes \overline{\eta_{s}}\right\|_{p}^{p}=\|f\chi_{E}\|_{p}^{p}\leq \|f\|_{p}^{p}.\]
So if $\varepsilon<\frac{\delta}{2|E|}$ the claim follows.
\end{proof}

	The following Lemma will allow us to get the lower bound we need and is similar to Lemma 7.8 in \cite{Voi}.

\begin{lemma}\label{L:ortho2} Let $H$ be a Hilbert space, and $\eta_{1},\cdots,\eta_{k}$ an orthonormal system in $H,$ and $V=\Span\{\eta_{j}:1\leq j\leq k\}$ and $P_{V}$ the projection onto $V.$   Let $K$ be a Hilbert space and $T\in B(H,K)$ with $\|T\|\leq 1.$ Then 
\[d_{\varepsilon}(\{T(\eta_{1}),\cdots,T(\eta_{k})\})\geq -k\varepsilon+\Tr(P_{V}T^{*}TP_{V}).\]
\end{lemma}

\begin{proof}  For a subspace $E\subseteq H$ we let $P_{E}$ be the projection onto $E.$ Let $W$ be a subspace of minimal dimension which $\varepsilon$-contains $\{T(\eta_{1}),\cdots,T(\eta_{k})\}.$  Then

\[\Tr(P_{W}TT^{*})=\Tr(P_{W}TT^{*}P_{W})\leq \Tr(P_{W}),\]

similarly
\begin{align*}
\Tr(P_{W}TT^{*})&\geq\Tr(P_{V}T^{*}P_{W}TP_{V})\\
&=\sum_{j=1}^{k}\ip{P_{W}T(\eta_{j}),T(\eta_{j})}\\
&\geq-\varepsilon k+\sum_{j=1}^{k}\ip{T(\eta_{j}),T(\eta_{j})}\\
&= -\varepsilon k+\Tr(P_{V}T^{*}TP_{V}).
\end{align*}

\end{proof}

For convenience, we shall identify $L(\Gamma)$ as a set of vectors in $l^{2}(\Gamma).$ That is, we shall consider $L(\Gamma)$ to be all $\xi\in l^{2}(\Gamma)$ so that 
\[\|\xi\|_{L(\Gamma)}=\sup_{\substack{f\in c_{c}(\Gamma),\\ \|f\|_{2}\leq 1}}\|\xi*f\|_{2}<\infty.\]

	Here $\xi*f$ is the usual convolution product. By standard arguments, if $\xi\in L(\Gamma),$ then for all $f\in l^{2}(\Gamma),$ $\xi*f\in l^{2}(\Gamma)$ and 
\[\|\xi*f\|_{2}\leq \|\xi\|_{L(\Gamma)}\|f\|_{2}.\]

	By general theory, $L(\Gamma)$ is closed under convolution and 
\[(\xi*\eta)*\zeta=\xi*(\eta*\zeta)\]
for $\xi,\eta,\zeta\in L(\Gamma).$ Finally for $\xi\in L(\Gamma),$ we set 
\[\xi^{*}(x)=\overline{\xi(x^{-1})}.\]

	If $\xi\in L(\Gamma),\zeta,\eta\in l^{2}(\Gamma),$ then
\[\ip{\xi*\eta,\zeta}=\ip{\eta,\xi^{*}*\zeta}.\]

	Finally, for $\xi\in L(\Gamma),f\in c_{c}(\Gamma),$  
\[\|f*\xi\|_{2}=\|\xi^{*}*f^{*}\|_{2}\leq \|f^{*}\|_{2}\|\xi^{*}\|_{L(\Gamma)}=\|f\|_{2}\|\xi\|_{L(\Gamma)}.\]

	Hence every element of $L(\Gamma)$ is bounded as a right convolution operator.

	We shall need a few more lemmas, for the first we require the following definitions.

\begin{definition}\label{D:polys}\emph{We let $\CC^{*}\ip{X_{1},\cdots,X_{n}}$ be the free $*$-algebra in $n$ noncommuting variables. That is  $\CC^{*}\ip{X_{1},\cdots,X_{n}}$ is the universal $\CC$-algebra generated by elements $X_{1},\cdots,X_{n},X_{1}^{*},\cdots,X_{n}^{*},$ and we equip  $\CC^{*}\ip{X_{1},\cdots,X_{n}}$ with a $*$-algebra structure defined on words (and extended by conjugate linearity) by}
\[(Y_{1}\cdots Y_{l})^{*}=Y_{l}^{*}\cdots Y_{1}^{*},\mbox{$Y_{j}\in \{X_{1},\cdots,X_{n},X_{1}^{*},\cdots,X_{n}^{*}\}$},\]
\emph{here $(X_{j}^{*})^{*}=X_{j}.$ We call elements of $\CC^{*}\ip{X_{1},\cdots,X_{n}}$ $*$-polynomials in $n$ noncommuting variables. Note that if $A$ is a $*$-algebra, and $a_{1},\cdots,a_{n}\in A,$ then there is a unique $*$-homomorphism $\CC^{*}\ip{X_{1},\cdots,X_{n}}\to A$ sending $X_{j}$ to $a_{j}.$ For $P\in \CC^{*}\ip{X_{1},\cdots,X_{n}},$ we denote the image under this homomorphism by $P(a_{1},\cdots,a_{n}).$} \end{definition}

\begin{definition}\emph{A} tracial $*$-algebra \emph{is a pair $(A,\tau)$ where $A$ is a unital $*$-algebra, $\tau\colon A\to \CC$ is a linear map so that $\tau(1)=1,\tau(x^{*}x)\geq 0,$ with $\tau(x^{*}x)=0$ if and only if $x=0,$ and $\tau(xy)=\tau(yx)$ for all $x,y\in A,$  and for all $x\in A,$ there is a $M>0$ so that $\tau(y^{*}x^{*}xy)\leq M\tau(y^{*}y)$ for all $y\in A.$ An} embedding sequence of $(A,\tau)$ \emph{is a sequence of maps $\sigma_{i}\colon A\to M_{d_{i}}(\CC)$ such that}
\[\sup_{i}\|\sigma_{i}(x)\|_{\infty}<\infty,\mbox{ where $\|\cdot\|_{\infty} $ is the operator norm, for all $x\in A,$}\]
\[\sigma_{i}(1)=1,\]
\[\frac{1}{n}\Tr(\sigma_{i}(x))\to \tau(x),\]
\[\|\sigma_{i}(P(x_{1},\cdots,x_{n}))-P(\sigma_{i}(x_{1}),\cdots,\sigma_{i}(x_{n}))\|_{2}\to 0\]
\emph{for all $x_{1},\cdots,x_{n}\in A,$ and $*$-polynomials $P$ in $n$ noncomuting variables. Here $\|x\|_{2}=\tau(x^{*}x)^{1/2}$ for $x\in A.$ We let $L^{2}(A,\tau)$ be the completion of $A$ in $\|\cdot\|_{2}.$ We also let $\pi_{\tau}\colon A\to B(L^{2}(A,\tau))$ be given by $\pi_{\tau}(x)a=xa,$ for $x,a\in A.$} \end{definition}

	The main example which will be relevant for us is $A=c_{c}(\Gamma)$ with the product being convolution and the $*$-being defined by consider $c_{c}(\Gamma)\subseteq L(\Gamma),$ and $\tau(f)=f(e).$ Then an embedding sequence of $\Gamma$ extends to one of $c_{c}(\Gamma)$ by
\[\sigma_{i}(f)=f(e)\id+\sum_{g\in \Gamma\setminus\{e\}}f(g)\sigma_{i}(g).\]

	We note that for the next Lemma, and throughout the paper we will use measure theoretic notation for certain norms on tracial von Neumann algebras $(M,\tau)$ . Thus $\|x\|_{\infty}$ will be the operator norm of $x,$ and $\|x\|_{p}=\tau((x^{*}x)^{p/2})^{1/p}.$

\begin{lemma} Let $(A,\tau)$ be a tracial $*$-algebra. And let $M$ be the weak operator topology closure of $\pi_{\tau}(A)$ equipped with the trace $\tau(x)=\ip{x1,1}$ for all $x\in M.$ Then any embedding sequence of $A$ extends to one of $M.$
\end{lemma}

\begin{proof} By standard arguments, $\tau$ is indeed a trace, since $A$ is $\|\cdot\|$-dense in $L^{2}(A,\tau),$ and elements of $M$ commute with right multiplication it follows that $\tau(x^{*}x)=0$ for $x\in M$ if and only if $x=0.$ If $x\in M\setminus A,$ by the Kaplansky Density Theorem we may choose a sequence $a_{n,x}$ so that $\|\pi_{\tau}(a_{n,x})\|_{\infty}\leq \|x\|_{\infty}$ and 
\[\|a_{n,x}-x\|_{2}<2^{-n}.\]

	Note that if $\omega$ is a free ultrafilter on $\NN,$ then $\sigma$ gives a trace-preserving embedding of $A$ into 
\[N=\{(x_{i}):x_{i}\in M_{d_{i}}(\CC),\sup_{n}\|x_{i}\|_{\infty}<\infty\}/\left\{(x_{i}):\lim_{n\to \omega}\frac{1}{n}\Tr(x_{i}^{*}x_{i})=0\right\},\]
where $N$ has the trace
\[\tau_{\omega}(x)=\lim_{n\to \omega}\frac{1}{d_{i}}\Tr(x_{i}),\]
if $x=(x_{i}).$  Thus $\sigma\big|_{\{a\in A:\|a\|_{\infty}\leq 1\}}$ is strong operator topology-strong operator topology continuous, and hence has an extension
\[\tau\colon \{a\in M:\|a\|_{\infty}\leq 1\}\to N.\]
	If we define
\[\rho(a)=\tau\left(\frac{a}{\|a\|_{\infty}}\right)\|a\|_{\infty},\]
it follows that $\rho$ is a trace-preserving $*$-homomorphism $M\to N.$ 

So by a standard contradiction and ultrafilter argument, for all $a\in A,$ we may find $a_{i}\in M_{d_{i}}(\CC)$ so that $\|a_{i}\|_{\infty}\leq \|\pi_{\tau}(a)\|_{\infty}$ and $\|a_{i}-\sigma_{i}(a)\|_{2}\to 0.$ 

For $x\in M,$ choose integers $1\leq i_{1}<i_{2}<i_{3}<\cdots,$ and elements $b_{n,x,i}\in M_{d_{i}}(\CC)$ so that $\|b_{n,x,i}\|_{\infty}\leq \|x\|_{\infty}$ and 
\[\|b_{j,x,i}-\sigma_{i}(a_{j,x})\|_{2}<2^{-n}\mbox{ for $1\leq j\leq n,i\geq i_{n},$}\]
\[\|\sigma_{i}(a_{j,x})-\sigma_{i}(a_{k,x})\|_{2}<2^{-n}+\|a_{j,x}-a_{k,x}\|_{2}\mbox{ for $1\leq j,k\leq n,i\geq i_{n},$}\]
the last inequality being possibile since $\sigma_{i}$ is an embedding sequence on $A.$ 

	For $x\in M\setminus A,$ define $\sigma_{i}(x)=b_{n,x,i}$ where $n$ is such that $i_{n}\leq i<i_{n+1}.$ If $x\in M\setminus A,$ and $i\geq i_{n}$ and $N$ is such that $i_{N}\leq i\leq i_{N+1},$ then
\begin{align*}
\|\sigma_{i}(x)-\sigma_{i}(a_{n,x})\|_{2}&\leq 2^{-n}+\|\sigma_{i}(a_{N,x})-\sigma_{i}(a_{n,x})\|\\
&\leq 2\cdot 2^{-n}+\|a_{N,x}-a_{i,x}\|_{2}\\
&\leq 4\cdot 2^{-n},
\end{align*}
\[\|\sigma_{i}(x)\|_{\infty}\leq \|x\|_{\infty}.\]

	From this estimate it is not hard to see that $\sigma_{i}$ is an embedding sequence of $M.$

\end{proof}

\begin{lemma}\label{L:soficops} Let $\Gamma$ be a countable sofic group, and $\Sigma=(\sigma_{i}\colon \Gamma\to S_{d_{i}})$ a sofic approximation of $\Gamma.$ Extend $\sigma_{i}$ to a embedding sequence, still denoted $\sigma_{i},$ of $(L(\Gamma),\tau)$ with $\tau$ the group trace. For $r,s\in \NN$ define $\sigma_{i}\colon M_{h,s}(L(\Gamma))\to M_{h,s}(M_{d_{i}}(\CC))$ by $\sigma_{i}(A)=[\sigma_{i}(a_{lr})]_{1\leq l\leq h,1\leq r\leq s}.$  Fix $n\in \NN.$ For $1\leq j\leq d_{i},1\leq k\leq n$ and $E\subseteq \Gamma$ finite define $T_{j,k}^{(E)}\colon l^{p}(\Gamma)^{\oplus n}\to l^{p}(d_{i})$ by
\[T_{j,k}^{(E)}(f)=\sum_{g\in E}f_{k}(g)\sigma_{i}(g)e_{j}.\]

	Then

(a) For all $E$ and $(1-o(1))nd_{i}$ of the $j,k$ we have $\|T^{(E)}_{j,k}\|_{l^{p}\to l^{p}}\leq 1$ as $i\to \infty.$

(b) For $1\leq p\leq \infty,$ for  all $\varepsilon>0,$ for all $f\in c_{c}(\Gamma),g\in l^{p}(\Gamma)^{\oplus n},$ there is a finite subset $E\subseteq \Gamma,$ so that if $E'\supseteq E$ is a finite subset of $\Gamma,$ then the set of $(j,k)$ so that 
\[\|T^{(E')}_{j,k}(f*g)-\sigma_{i}(f)T^{(E)}_{j,k}(g)\|_{p}\leq \varepsilon\|g\|_{p},\]
	has cardinality at least  $(1-\varepsilon))nd_{i}$ for all large $i.$ 

(c) For all $\varepsilon>0,$ for all $\xi\in M_{1,n}(L(\Gamma)), $ (identifying $M_{1,n}(L(\Gamma))$ as a subset of $l^{2}(\Gamma)^{\oplus n}$)  there is a finite subset $E\subseteq \Gamma,$ so that if $E'\supseteq E$ is a finite subset of $\Gamma,$ then the set of $(j,k)$ so that 
\[\|T^{(E')}_{j,k}(\xi)-\sigma_{i}(\xi)(e_{j}\otimes e_{k})\|_{2}<\varepsilon,\]
(here $e_{j}\otimes e_{k}\in l^{2}(d_{i})^{\oplus n}$ is $e_{j}$ in the $k^{th}$ coordinate and zero otherwise).
	has cardinality at least $(1-\varepsilon)nd_{i}$ for all large $i.$

\end{lemma}

\begin{proof}

(a) We have
\[\left\|T^{(E)}_{j,k}(f)\right\|_{p}^{p}=\sum_{r=1}^{d_{i}}\left|\sum_{\substack{g\in E,\\ \sigma_{i}(g)(j)=r}}f_{k}(g)\right|^{p}.\]

	Let $C_{i}=\{j\in \{1,\cdots,d_{i}\}:\sigma_{i}(g)(j)\ne \sigma_{i}(h)(j)\mbox{ for $g\ne h$ in $E$}\}.$ By soficity, we have $\frac{|C_{i}|}{d_{i}}\to 1,$ and if $j\in C_{i}$ we have
\[\left\|T^{(E)}_{j,k}(f)\right\|_{p}^{p}\leq \|f_{k}\|_{p}^{p}\leq \|f\|_{p}^{p}.\]

(b) For $A\in M_{d_{i}}(\CC),$
\[\|A\|_{2}^{2}=\frac{1}{d_{i}}\sum_{j=1}^{d_{i}}\|Ae_{j}\|_{2}^{2},\]
where $e_{j}$ is the vector which has $j^{th}$ coordinate equal to $1,$ and all other coordinates zero. Hence by Chebyshev's inequality, the fact that $\|T^{(E)}_{j,k}\|_{p}\leq 1,$ and the definition of embedding sequences, it is enough to verify this for $f=\delta_{x},g=\delta_{y}$ for some $x,y\in\Gamma.$ But this is trivial from the definition of soficity.

(c) Let us first verfiy this when $\xi\in M_{1,n}(c_{c}(\Gamma)).$ In this case, we may again reduce to $\xi=(\delta_{a_{1}},\cdots,\delta_{a_{k}})$ for some $a_{1},\cdots,a_{k}\in \Gamma.$ Then if $E\supseteq \{a_{1},\cdots,a_{k}\}$ we have
\[T^{(E)}_{j,k}(\xi)=\sigma_{i}(a_{k})e_{j}=\sigma_{i}(\xi)(e_{j}\otimes e_{k}).\]

	In the general case let $\varepsilon>0,$ given $\xi\in M_{1,n}(L(\Gamma))$ choose $f\in M_{1,n}(c_{c}(\Gamma))$ so that $\|f-\xi\|_{2}<\varepsilon.$  Thus for $(1-(\varepsilon+o(1)))kd_{i}$ of the $(j,k)$ we have
\[\|T^{(E')}_{j,k}(\xi)-\sigma_{i}(\xi)(e_{j}\otimes e_{k})\|_{2}\leq 2\varepsilon+\|(\sigma_{i}(\xi)-\sigma_{i}(f))(e_{j}\otimes e_{k})\|.\]

	By the definition of embedding sequence for all large $i$ we have
\[\frac{1}{d_{i}}\sum_{j=1}^{d_{i}}\sum_{k=1}^{n}\|(\sigma_{i}(\xi)-\sigma_{i}(f))(e_{j}\otimes e_{k})\|_{2}^{2}<\varepsilon^{2},\]

	thus for at least $(1-\sqrt{\varepsilon})nd_{i}$ of the $(j,k)$ we have
\[\|(\sigma_{i}(\xi)-\sigma_{i}(f))(e_{j}\otimes e_{k})\|_{2}<\sqrt{\varepsilon},\]
combining these estimates completes the proof.

\end{proof}

	We need a similar lemma for $\R^{\omega}$-embeddable groups.

\begin{lemma}\label{L:op2} Let $\Gamma$ be a countable $\R^{\omega}$-embeddable group, and $\Sigma=(\sigma_{i}\colon \Gamma\to U(d_{i}))$ an embedding sequence. Define $\rho_{i}\colon \Gamma \to U(S^{2}(d_{i}))$ by $\rho_{i}(g)A=\sigma_{i}(g)A\sigma_{i}(g)^{-1}.$ Extend $\sigma_{i},\rho_{i}$ to  embedding sequences, still denoted $\sigma_{i},\rho_{i}$ of $(L(\Gamma),\tau)$ with $\tau$ the group trace. For $h,s\in \NN$ define $\sigma_{i}\colon M_{h,s}(L(\Gamma))\to M_{h,s}(M_{d_{i}}(\CC))$ by $\sigma_{i}(A)=[\sigma_{i}(a_{lr})]_{1\leq l\leq h,1\leq r\leq s}.$  Fix $n\in \NN.$ For $\xi,\eta\in l^{2}(d_{i}),1\leq k\leq d_{i}$ and $E\subseteq \Gamma$ finite define $T_{\xi,\eta,k}^{(E)}\colon l^{p}(\Gamma)^{\oplus n}\to S^{p}(d_{i})$ by
\[T_{\xi,\eta,k}^{(E)}(f)=\sum_{g\in E}f_{k}(g)\sigma_{i}(g)\xi \otimes \overline{\sigma_{i}(g)\eta}.\]

	Then

(a) There exists measurable $A_{i}\subseteq S^{2d_{i}-1}$ with $\mathbb P(A_{i})\to 1,$ so that 
\[\{(\xi,\eta)\in (S^{2d_{i}-1})^{2}:\|T_{\xi,\eta,k}^{(E)}\|_{l^{p}\to S^{p}}\leq 2\}\supseteq A_{i}\times A_{i},\]
for $(1-o(1))d_{i}$ of the $k.$

(b) For all $\varepsilon>0,$ for all $f\in c_{c}(\Gamma),g\in l^{p}(\Gamma)^{\oplus n},$ there exists measurable $B_{i}\subseteq S^{2d_{i}-1},$ with $\mathbb P(B_{i})\geq 1-\varepsilon,$ for all large $i,$ a finite subset $E\subseteq \Gamma,$ so that if $E'\supseteq E$ is a finite subset of $\Gamma,$ then for $(1-\varepsilon)d_{i}$ of the $k$ and for all large $i,$
\[\{(\xi,\eta)\in (S^{2d_{i}-1})^{2}:\|T^{(E')}_{\xi,\eta,k}(f*g)-\rho_{i}(f)T^{(E)}_{\xi,\eta,k}(g)\|_{p}<\varepsilon\}\supseteq B_{i}\times B_{i}\]

(c) For all $\varepsilon>0,$ for all $\zeta\in M_{1,n}(L(\Gamma)), $ (identifying $M_{1,n}(L(\Gamma))$ as a subset of $l^{2}(\Gamma)^{\oplus n}$  ) there are measurable $C_{i}\subseteq S^{2d_{i}-1},$ with $\PP(C_{i})\geq 1-\varepsilon$ for all large $i,$ a finite subset $E\subseteq \Gamma,$ so that if $E'\supseteq E$ is a finite subset of $\Gamma,$ so that for at least $(1-\varepsilon)d_{i}$ of the $k$ and for all large $i,$
\[\{(\xi,\eta)\in (S^{2d_{i}-1})^{2}:\|T^{(E')}_{\xi,\eta,k}(\zeta)-\rho_{i}(\zeta)\xi\otimes \overline{\eta})\|_{2}<\varepsilon\}\supseteq C_{i}\times C_{i},\]
	has cardinality at least $(1-\varepsilon)nd_{i}$ for all large $i.$

\end{lemma}

\begin{proof} Same as the preceding Lemma, but using Lemma \ref{L:choose}.

\end{proof}

	Finally we need one last lemma, which allows us to reduce to considering subspaces of finite direct sums of $l^{p}(\Gamma).$

\begin{lemma}\label{L:vnDcont} Let $\Gamma$ be a countable discrete group. Let $H\subseteq l^{2}(\NN,l^{2}(\Gamma))$ be a closed $\Gamma$-invariant subspace.

(a) Define $\pi_{k}\colon l^{2}(\NN,l^{2}(\Gamma))\to l^{2}(\Gamma)^{\oplus k}$  by $\pi_{k}f(j)=f(j)$ for $1\leq j\leq k.$ Then
\[\dim_{L(\Gamma)}(H)=\sup_{k}\dim_{L(\Gamma)}(\overline{\pi_{k}(H)}^{\|\cdot\|_{2}}).\]

(b) The representation $H$ is isomorphic to a direct sum of representations of the form $l^{2}(\Gamma)p$ with $p\in L(\Gamma)$ (by the remarks preceding definition \ref{D:polys} each element of $L(\Gamma)$ is a bounded right convolution operator) an orthogonal projection.
\end{lemma}

\begin{proof}

(a)	Since $\pi_{k}(H)$ is dense in $\overline{\pi_{k}(H)}$ we have
\[\dim_{L(\Gamma)}(H)\geq\sup_{k}\dim_{L(\Gamma)}(\overline{\pi_{k}(H)}^{\|\cdot\|_{2}}).\]

	Let us first handle the case when $\dim_{L(\Gamma)}(H)<\infty,$ let $P$ be the projection onto $H.$ 

	Then
\begin{align*}
\dim_{L(\Gamma)}(\overline{\pi_{k}(H)})&=\dim_{L(\Gamma}(\ker(\pi_{k}P)^{\perp})\\
&=\dim_{L(\Gamma)}(H\cap\overline{(H^{\perp}+l^{2}(\Gamma)^{\oplus k}}))\\
&=\dim_{L(\Gamma)}(H\cap (H\cap l^{2}(\NN\setminus \{1,\cdots,k\},\Gamma))^{\perp})).
\end{align*}

	Let $Q_{k}$ be the projection onto $H\cap l^{2}(\NN\setminus \{1,\cdots,k\},\Gamma).$ Then
\begin{align*}
\dim_{L(\Gamma)}(H\cap l^{2}(\NN\setminus \{1,\cdots,k\},\Gamma))&= \sum_{n=1}^{\infty}\ip{Q_{k}(\delta_{e}\otimes e_{n}),\delta_{e}\otimes e_{n}}\\
&=\sum_{n=k}^{\infty}\ip{Q_{k}(\delta_{e}\otimes e_{n}),\delta_{e}\otimes e_{n}}\\
&\leq \sum_{n=k}^{\infty}\ip{P(\delta_{e}\otimes e_{n}),\delta_{e}\otimes e_{n}}\\
&\to 0,
\end{align*}
as $\dim_{L(\Gamma)}(H)<\infty.$

	In the general case, it suffices to show that we may write $H$ as a direct sum of representations with finite von Neumann dimension. Zorn's Lemma implies that every representation is a direct sum of cyclic representations which are contained in $l^{2}(\NN,l^{2}(\Gamma))$, so it suffices to show every cyclic representation contained in $l^{2}(\NN,l^{2}(\Gamma))$ has finite von Neumann dimension.

	For this, let $\xi\in H$ be a cyclic vector, then by Theorem V.3.15 in \cite{Taka} there is vector $\zeta\in l^{2}(\Gamma)$ so that 
\[\ip{g\xi,\xi}=\ip{g\zeta,\zeta}\]
for all $g\in \Gamma.$ Thus $H$ is isomorphic to $\overline{\Span}^{\|\cdot\|_{2}}(\Gamma\xi)$ via the unitary sending $g\xi\to g\zeta.$ From this it clear that $H$ has dimension at most $1.$ 

(b) As in part $(a),$ we may assume that $H$ is a cyclic representation contained in $l^{2}(\Gamma).$ Let $p$ be the projection onto $H,$ then $p$ commutes with $L(\Gamma).$ Set $\xi=p(\delta_{e}),$ since $p$ commutes with $L(\Gamma),$ it is not hard to see that $p(f)=f*\xi$ for $f\in c_{c}(\Gamma).$ Arguments entirely similar to those before Definition \ref{D:polys} prove that $\xi$ is a bouned left convolution operator. Hence $\xi$ is an orthogonal projection in $L(\Gamma),$ and $H=l^{2}(\Gamma)\xi.$  

\end{proof}

\begin{theorem}\label{T:lowerbound} Let $\Gamma$ be a countable discrete group, let $1\leq p\leq 2,$ and $Y$ a closed $\Gamma$-invariant subspace of $l^{p}(\NN,l^{p}(\Gamma)),$ with $\Gamma$ acting by $gf(x)=f(g^{-1}x).$  Set $H=\overline{Y}^{\|\cdot\|_{2}}.$

(a) Suppose $\Sigma$ is a sofic approximation of $\Gamma,$ then
\[\underline{\dim}_{\Sigma,l^{p}}(Y,\Gamma)\geq \dim_{L(\Gamma)}(H).\]

(b)Suppose $\Sigma$ is an embedding sequence of $\Gamma,$ then
\[\underline{\dim}_{\Sigma,S^{p},\mbox{conj}}(Y,\Gamma)\geq \dim_{L(\Gamma)}(H).\]

(c) Suppose $\Sigma$ is an embedding sequence of $\Gamma,$ and $H\subseteq l^{2}(\NN,l^{2}(\Gamma))$ is $\Gamma$ invariant, then
\[\underline{\dim}_{\Sigma,l^{2}}(H,\Gamma)\geq \dim_{L(\Gamma)}(H).\]

\end{theorem}

\begin{proof} We first reduce to the case that $Y\subseteq l^{p}(\Gamma)^{\oplus h}$ with $h$ finite.

Consider the projection
\[\pi_{h}\colon l^{p}(\NN,\Gamma)\to l^{p}(\{1,\cdots,h\},l^{p}(\Gamma))\] given by
\[\pi_{h}f(j)=f(j),\]
assume we know the result for $Y\subseteq l^{p}(\Gamma)^{\oplus h}$ for each $h.$ 

	Then,
\begin{align*}
\dim_{\Sigma,l^{p}}(Y,\Gamma)&\geq \dim_{\Sigma,l^{p}}(\overline{\pi_{h}(Y)}^{\||\cdot\|_{p}},\Gamma)\\
&\geq \dim_{L(\Gamma)}(\overline{\pi_{h}(H)}^{\|\cdot\|_{2}}),
\end{align*}
letting $h\to \infty$ and applying the preceding Lemma proves the claim. Thus, we shall assume that $Y\subseteq l^{p}(\Gamma)^{\oplus n}$ with $n\in \NN.$

	By part (b) of the preceding Lemma, we can find vectors $(\xi^{(q)})_{q=1}^{\infty}\in H,$ so that 
\[\ip{\lambda(g)\xi^{(s)},\xi^{(s)}}=\ip{\lambda(g) q_{s},q_{s}}=q_{s}(g^{-1}),\mbox{ where $q_{s}$ is a projection in $L(\Gamma)$},\]
\[\sum_{s=1}^{\infty}\tau(q_{s})=\dim_{L(\Gamma)}(H),\]
\[\ip{\lambda(g)\xi^{(j)},\xi^{(l)}}=0\mbox{ for $j\ne l,g\in \Gamma.$}\]
\[H=\bigoplus_{j=1}^{\infty}\overline{L(\Gamma)\xi^{(j)}}.\]

	These equations can be rewritten as 
\[\sum_{i=1}^{n}\xi^{(j)}*(\xi^{(j)})^{*}=q_{j},\mbox{ for $1\leq j\leq \infty$}\]
\[\sum_{i=1}^{n}\xi^{(j)}*(\xi^{(l)})^{*}=0\mbox{ if $j\ne l$},\]

	Let us illuminate these equations a little. Regard a vector $\xi\in l^{2}(\Gamma)^{\oplus n}$ as a element in $M_{1,n}(l^{2}(\Gamma))$ with the product of two matrices induced from convolution of vectors. Then the product of elements of $M_{1,n}(l^{2}(\Gamma)),M_{n,1}(L(\Gamma))$ makes sense, but may not land back in $l^{2}(\Gamma).$ The above equations then read
\[\xi^{(j)}(\xi^{(j)})^{*}=q_{j}\mbox{for $1\leq j<\infty$},\]
\[\xi^{(j)}(\xi^{(l)})=0\mbox{ for $j\ne l.$}.\]

	In particular, the above equations imply that 
\[\|\xi^{(j)}_{r}\|_{L(\Gamma)}\leq 1.\]
 So that $\xi^{(j)}\in M_{1,n}(L(\Gamma)).$ Extend $\sigma_{i}$ to a embedding sequence of $M_{n,m}(L(\Gamma))$ for all $n,m$  and such that 
\[\|\sigma_{i}(\xi^{(j)})\|\leq 1,\mbox{ for all $j$}\]
\[\|\sigma_{i}(\xi^{(j)}_{r})\|\leq 1,\mbox{ for all $j,r$}\]
\[\sigma_{i}(\xi^{(j)})\sigma_{i}(\xi^{(l)})^{*}=0\mbox{ for all $j\ne l.$}\]
for all $j,r.$

(a) 

	Let $S=(x_{j})_{j=1}^{n}$ be a dynamical generating sequence for $Y.$ 

 Fix $\eta>0,t\in \NN$ and choose a finite subset $F_{1}\subseteq \Gamma,m_{1}\in \NN,$ and $c_{gj}^{(s)}$ for $1\leq s\leq t,(g,j)\in F_{1}\times \{1,\cdots,m_{1}\}$ so that for all $1\leq s\leq t$
\[\left\|\xi^{(s)}-\sum_{\substack{g\in F_{1}\\ 1\leq j\leq m_{1}}}c_{gj}^{(s)}gx_{j}\right\|_{2}<\eta.\]

	Choose finitely supported functions $x_{j}'$ so that $\|x_{j}-x_{j}'\|_{p}<\eta'$ . Since $p\leq 2,$ it is easy to see that if we  force $\eta'$ to be sufficiently small then,
\[\left\|\xi^{(s)}-\sum_{\substack{g\in F_{1}\\ 1\leq j\leq m_{1}}}c_{gj}^{(s)}gx_{j}'\right\|_{2}<\eta.\]

Let $S=(x_{j})_{j=1}^{\infty}$ be a dynamically generating sequence for $Y.$ Fix $F\subseteq \Gamma$ finite $m\in \NN,\delta>0.$ Let $E\subseteq \Gamma$ be finite, let $T_{j,k}^{(E)}$ be defined as  Lemma \ref{L:soficops}.

	It is easy to see that if $E$ is sufficently large, then $T_{j,k}^{(E)}\big|_{Y_{F,m}}\in \Hom_{\Gamma}(S,F,m,\delta,\sigma_{i})_{2}$ for $(1-o(1))nd_{i}$ of the $j,k,$ and  in fact $\|T^{(E)}_{j,k}\|_{l^{p}\to l^{p}}\leq 2$ for $1\leq p\leq 2.$  For such $(j,k),$ and for all small $\delta,$ for $1\leq s\leq t+1$

\[\left\|T_{j,k}^{(E)}(\xi^{(s)})-\sum_{\substack{g\in F_{1}\\ 1\leq j\leq m_{1}}}c_{gj}^{(p)}\sigma_{i}(g)T_{j,k}^{(E)}(x_{j})\right\|_{2}<2\eta,\]
\[\|T_{j,k}^{(E)}(gx_{j}')-T_{j,k}^{(E)}(gx_{j})\|_{2}<\eta.\]

	Thus by Lemma \ref{L:soficops} for at least $(1-(2013)!\varepsilon)nd_{i}$ of the $j,k$ we have
\[\left\|\sigma_{i}(\xi^{(s)})(e_{j}\otimes e_{k})-\sum_{\substack{g\in F_{1}\\ 1\leq j\leq m_{1}}}c_{gj}^{(p)}\sigma_{i}(g)T_{j,k}^{(E)}(x_{j})\right\|_{2}<\varepsilon+\eta.\]

	Now consider the linear map $A\colon l^{\infty}(\NN,l^{p}(d_{i}))\to l^{2}(d_{i})^{\oplus t}$ given by
\[S(f)=\left(\sum_{\substack{g\in F_{1}\\ 1\leq j\leq m_{1}}}c_{gj}^{(p)}\sigma_{i}(g)f(j)\right)_{p=1}^{t},\]
from the above it is easy to see that if $\alpha_{S}(\Hom_{\Gamma}(S,F,m,\delta,\sigma_{i}))\subseteq_{\varepsilon'}V$ and $\varepsilon'$ is sufficiently small, 
\[A(V)\supseteq_{\varepsilon,\|\cdot\|_{2}}\{\phi_{i}(e_{j}\otimes e_{k}):(j,k)\in A_{i}\},\]
with 
\[\frac{|A_{i}|}{d_{i}}\to (1-(2013)!\varepsilon)nd_{i},\]
\[\phi_{i}(f)=(\sigma_{i}(\xi^{(1)})(f),\sigma_{i}(\xi^{(2)})(f),\cdots,\sigma_{i}(\xi^{(t)})(f)).\]

	Thus $\phi_{i}$ is given in matrix form by
\[\phi_{i}=\begin{bmatrix}
\sigma_{i}(\xi^{(1)})&0&\cdots&0\\
0&\sigma_{i}(\xi^{(2)})&\cdots&0\\
\vdots&\ddots&\cdots&\vdots\\
0&0&\cdots&\sigma_{i}(\xi^{(t)})
\end{bmatrix}.\]

	As
\[\phi_{i}\phi_{i}^{*}=\begin{bmatrix}
\sigma_{i}(\xi^{(1)})\sigma_{i}(\xi^{(1)})^{*}&0&\cdots&0\\
0&\sigma_{i}(\xi^{(2)})\sigma_{i}(\xi^{(2)})^{*}&\cdots&0\\
\vdots&\ddots&\cdots&\vdots\\
0&0&\cdots&\sigma_{i}(\xi^{(t)})\sigma_{i}(\xi^{(t)})^{*}
\end{bmatrix}\]

	By our choice of $\sigma_{i}$ we have
\[\|\phi_{i}\|\leq 1,\]

	By Lemma \ref{L:ortho2}, we find that 
\[\dim_{\Sigma,l^{p}}(V,\Gamma)\geq (1-(2013)!\varepsilon)n+\dim_{L(\Gamma)}H_{t}.\]

	Letting $\varepsilon\to 0,t\to \infty$ completes the proof.

(b), (c) Same proof as in $(a),$ one instead uses Lemma \ref{L:op2}, Lemma $\ref{L:choose},$ and the formula
\[\mathbb P(A)=\int_{U(d_{i})}\frac{|\{j:Ue_{j}\in A\}||{d_{i}}}\,dU,\]
for $A\subseteq S^{2d_{i}-1},$ to find an orthonormal system $\zeta_{1},\cdots,\zeta_{q}$ with $q\geq (1-\varepsilon)d_{i},$ so that $T_{\zeta_{j},\zeta_{p},k}^{(E)}\in \Hom_{\Gamma}(\cdots)$ for most $k$ and all $j,p.$
\end{proof}

\begin{cor}\label{T:lp} Let $1\leq p\leq 2,$  $V$ a finite-dimensional normed vector space,  and $\Gamma$ a countable discrete group. 

(a) If $\Gamma$ is sofic and $\Sigma$ is a sofic approximation of $\Gamma$, then
\[\underline{\dim}_{\Sigma,l^{p}}(l^{p}(\Gamma,V),\Gamma)=\dim_{\Sigma,l^{p}}(l^{p}(\Gamma,V),\Gamma)=\dim V.\]

(b) If $\Gamma$ is $\R^{\omega}$-embeddable and $\Sigma$ is an  embedding sequence of $\Gamma,$  then 

\[\underline{\dim}_{\Sigma,l^{2}}(l^{2}(\Gamma,l^{2}(n)),\Gamma)=\dim_{\Sigma,l^{2}}(l^{2}(\Gamma,l^{2}(n)),\Gamma)=n.\]
\[\underline{\dim}_{\Sigma,S^{p},conj}(l^{p}(\Gamma,V),\Gamma)=\dim_{\Sigma,S^{p},conj}(l^{p}(\Gamma,V),\Gamma)=\dim V.\]
\end{cor}

\begin{proof} The lower bounds are automatic from the preceding Theorem. The upper bounds are easy since $l^{p}(\Gamma,V)$ can be generated by $\dim V$ elements.

\end{proof}

\begin{cor} Let $\Gamma$ be a $\R^{\omega}$-embeddable group $1\leq p\leq 2.$ If $V,W$ are finite dimensional vector spaces with $\dim V<\dim W,$  then there are no $\Gamma$-equivariant bounded linear maps from $l^{p}(\Gamma,V)$ to $l^{p}(\Gamma,W)$ with dense image.  Consequently if $2\leq p<\infty,$ then there are no $\Gamma$-equivariant bounded linear injections from $l^{p}(\Gamma,W)$ to $l^{p}(\Gamma,V)$. \end{cor}
\begin{proof} For $1\leq p\leq 2$ this is immediate from the above corollary and Proposition \ref{P:surjection}. The other result follow by duality.
\end{proof}

\begin{theorem}\label{T:dim1} Let $\Gamma$ be a $\R^{\omega}$-embeddable group, and $\pi\colon \Gamma\to U(H)$ a representation, such that $\pi\leq \lambda^{\oplus \infty}.$ Then for every embedding sequence $\Sigma,$ 
\[\dim_{\Sigma,l^{2}}(\pi)=\underline{\dim}_{\Sigma,l^{2}}(\pi)=\dim_{L(\Gamma)}(\pi).\] 

\end{theorem}
\begin{proof} Let $\lambda\colon \Gamma\to \mathcal{U}(l^{2}(\Gamma))$ be given by $\lambda(g)f(x)=f(g^{-1}x).$ We already know from  Theorem $\ref{T:lp}$ that
\[\dim_{\Sigma,l^{2}}\lambda^{\oplus n}=\underline{\dim}_{\Sigma,l^{2}}\lambda^{\oplus n}=n.\]

Let us first assume that $\pi$ is cyclic with cyclic vector $\xi,$ then as in Lemma \ref{L:vnDcont} we may find a $\zeta\in l^{2}(\Gamma)$ so that
\[\ip{\pi(x)\xi,\xi}=\ip{\lambda(x)\zeta,\zeta},\]
so $\pi \leq \lambda.$  Let $\pi'$ be a representation such that $\lambda=\pi\oplus \pi',$ then by  Theorem \ref{T:lowerbound} we have
\begin{align*}
1=\dim_{\Sigma,l^{2}}\lambda &\geq \dim_{\Sigma,l^{2}}\pi+\underline{\dim}_{\Sigma,l^{2}}\pi'\\
&\geq\underline{\dim}_{\Sigma,l^{2}}\pi+\underline{\dim}_{\Sigma,l^{2}}\pi'\\
&\geq \dim_{L(\Gamma)}\pi+\dim_{L(\Gamma)}\pi'\\
&=1.
\end{align*}
Thus all the above inequalities must be equalities, in particular
\[\dim_{\Sigma,l^{2}}\pi=\underline{\dim}_{\Sigma,l^{2}}\pi=\dim_{L(\Gamma)}\pi.\]

	In the general case, apply Zorn's Lemma to write $\pi=\bigoplus_{n=1}^{\infty}\pi_{n}$ with $\pi_{n}$ cyclic. Then by Corollary $\ref{C:dinfsumadd}$
\[\underline{\dim}_{\Sigma,l^{2}}(\pi)\geq \sum_{n=1}^{\infty}\underline{\dim}_{\Sigma,l^{2}}(\pi_{n})=\sum_{n=1}^{\infty}\dim_{L(\Gamma)}\pi_{n}=\dim_{L(\Gamma)}\pi,\]
\[\dim_{\Sigma,l^{2}}(\pi)\leq \sum_{n=1}^{\infty}\dim_{\Sigma,l^{2}}(\pi_{n})=\sum_{n=1}^{\infty}\dim_{L(\Gamma)}\pi_{n}=\dim_{L(\Gamma)}\pi.\]
This completes the proof of the theorem.

\end{proof}

$\mathbf{Acknowledgments}.$ The author would like to thank Dimitri Shlyakhtenko and Yoann Dabrowski  for alerting him to this problem, and Dimitri Shlyakhtenko for his helpful advice on the problem.  The author would also like to thank Hanfeng Li for pointing out errors in a previous version of this paper.

\end{document}